\documentclass[11pt]{article} 
\usepackage{amsmath}
\usepackage{amssymb}

\usepackage[all]{xy}
\usepackage[protrusion=true]{microtype}
\usepackage[colorlinks=true,linkcolor=blue,citecolor=red]{hyperref}
\usepackage{xcolor}

\newtheorem{theorem}{Theorem}[section]
\newtheorem{proposition}[theorem]{Proposition}
\newtheorem{corollary}[theorem]{Corollary}
\newtheorem{definition}[theorem]{Definition}
\newtheorem{remark}[theorem]{Remark}
\newtheorem{example}[theorem]{Example}
\newenvironment{proof}{\begin{trivlist}\item\textit{Proof:}\ }{\hfill$\Box$\end{trivlist}}

\usepackage{a4wide}
\def\ppp{\mathbf{p}}
\def\ww{\mathbf{w}}
\def\zz{\mathbf{z}}
\def\HH{\mathbb{H}}
\def\RR{\mathbb{R}}
\def\CC{\mathbb{C}}
\def\ZZ{\mathbb{Z}}
\def\GG{\mathbb{G}}
\def\wtPhi{\widetilde{\Phi}}
\def\wtM{\widetilde{M}}
\def\CPn{\mathbb{CP}^n}
\def\CHn{\mathbb{CH}^n}
\newcommand{\su}{\mathfrak{su}}

\newcommand{\tPsi}{\tilde{\Psi}}
\newcommand{\xx}{\mathbf{x}}
\renewcommand{\ww}{\mathbf{w}}

\begin{document}
 
\title{\textbf{A Twistor Construction of Hopf Real Hypersurfaces in 
Complex hyperbolic Space}} 

\author{Jong Taek Cho, Makoto Kimura and Miguel Ortega
\footnote{J.~T.~Cho is supported by Basic Science Research Program through the National Research
Foundation of Korea(NRF) funded by the Ministry of Education, Science and Technology
(2019R1F1A1040829). 
 M.~Kimura is supported by JSPS KAKENHI Grant Number 
JP20K03575. 
M.~Ortega is partially supported by the projects 
PID2020-116126GB-I00 (MICINN), 
PY20-01391 (Junta de Andaluc\'ia)
and the “Maria de Maeztu” Excellence
Unit IMAG, reference CEX2020-001105-M, funded by
MCIN/AEI/10.13039/501100011033/ .
}} 

\date{\today}

\maketitle 

\begin{abstract} It is very well known that Hopf real hypersurfaces in the complex projective space can be locally characterized as tubes over complex submanifolds. This also holds true for some, but not all, Hopf real hypersurfaces in the complex hyperbolic space. The main goal of this paper is to show, in a unified way, how to construct Hopf real hypersurfaces in the complex hyperbolic space from a horizontal submanifold in one of the three twistor spaces of the indefinite complex $2$-plane Grassmannian with respect to the natural para-quaternionic K\"ahler structure. We also identify these twistor spaces with the sets of circles in totally geodesic complex hyperbolic lines in the complex hyperbolic space. As an application, we describe all classical Hopf examples. We also solve the remarkable and long-standing problem of the existence of Hopf real hypersurfaces in the complex hyperbolic space, different from the horosphere, such that the associated principal curvature is $2$. We exhibit a  method to obtain plenty of them. 
\end{abstract}

\noindent \emph{MSC:} Primary 53C40
\par
\noindent \emph{Keywords:} 
Complex hyperbolic Space, 
Hopf hypersurfaces, 
Twistor space, 
Indefinite Complex $2$-plane Grassmannian.

\section{Introduction}

A very important family of real hypersurfaces in a K\"ahler manifold $(\widetilde{M},J)$  consists of the so-called Hopf real hypersurfaces. A real hypersurface $M^{2n-1}$ in $\widetilde{M}^n$ is called \textit{Hopf} if for a (local) unit normal vector $N$, the structure vector field of $M$,  $\xi:=-JN$, is an eigenvector  of the shape operator $A$ of $M$, i.e., $A\xi=\mu\xi$. When the ambient space is the complex projective space $\CPn$ (of constant holomorphic sectional curvature $4$), such family is very nice, among other reasons, because (1) when its focal map has constant rank, the image of its focal map is locally a complex submanifold in $\CPn$, (2) any tube over a complex submanifold is Hopf according to Theorems 2.1 and 2.2 in \cite{CR}), and (3) any homogeneous real hypersurface in $\CPn$ is Hopf (see \cite{Tak1}, \cite{Tak2}). 
\par
For a Hopf hypersurface with $A\xi=\mu\xi$ in the complex hyperbolic space $\CHn$ of constant holomorphic sectional curvature $-4$, similar results to Theorems 2.1 and 2.2 in \cite{CR} hold true (cf. \cite{Mo}) provided $|\mu|>2$. In particular, a tube of a specific radius of such real hypersurface will be a complex submanifold. In addition, tubes over complex submanifolds satisfy $|\mu|>2$. However, this is not true any more when $|\mu|\le 2$. A simple counterexample is the horosphere ($\vert\mu\vert=2$), since any tube over it is again another horosphere (see \cite{Mo}), which is in the list of homogeneous real hypersurfaces in the complex hyperbolic space (cf. \cite{B}). Surprisingly, in this classification result, there are some examples which are not Hopf.  On the other hand, Hopf hypersurfaces with $|\mu|<2$ in $\CHn$ can be constructed from an arbitrary pair of Legendrian submanifolds in $S^{2n-1}$ (see \cite{I}, \cite{IR}). 
\par
The main target of this paper is to construct, in a unified way, Hopf real hypersurfaces in $\CHn$ for any $\mu$. The inspiring idea of our construction is the fact that given a Hopf real hypersurface in $\CHn$, an integral curve of $\xi$ is a Frenet curve of order 2 with constant curvature $\kappa$, and contained in a totally geodesic $\mathbb{CH}^1(-4)\subset \mathbb{CH}^n(-4)$. There are three families of such curves, depending on $\kappa=2$, $\kappa>2$ or $0\leq\kappa<2$, with different geometric behaviours.  More specifically, the family of \textit{circles} in $\mathbb{CH}^1(-4)$ is divided into $3$ classes (up to isometries): (i) a geodesic ($\kappa=0$) and its equidistant curves, ($0<\kappa<2$), (ii) a horocycle ($\kappa=2$), and (iii) a geodesic circle ($\kappa>2$). Thus, a second target of this paper, which is interesting by itself, is to determine the spaces of such curves in $\CHn$. These spaces will be used in the constructions of Hopf real hypersurfaces in $\CHn$. 
\par
In order to achieve our goals, we will use \textit{twistor spaces}.  Needless to say, they play important roles both in Geometry and Physics. For example, minimal $2$-spheres in $S^4$ are closely related to horizontal holomorphic curves with respect to the \textit{twistor fibration} $\mathbb{CP}^3\rightarrow$ $S^4$ (see \cite{Br}). This result is generalized to maximal K\"ahler submanifolds in quaternionic K\"ahler manifolds and horizontal complex Legendrian submanifolds in the twistor space (see \cite{AM2}). We will also need \textit{para-quaternionic K\"ahler manifolds}, which we will review in Section \ref{PQK}. 
\par 
Our starting point is $V_{1,1}(\mathbb{C}_1^{n+1})$, which is the \textit{indefinite complex Stiefel manifold of 
orthonormal time-like and space-like vectors in} $\CC_1^{n+1}$  (Section \ref{Twist}), which is a bundle over the indefinite complex 2-plane Grassmannian $M_G=\GG_{1,1}(\CC_1^{n+1})$. It turns out that $\GG_{1,1}(\CC_1^{n+1})$ is a para-quaternionic K\"ahler manifold, as it is shown in Section \ref{Twist}. We set three \textit{twistor spaces} $M_{Z}^s$, $s\in\{+,-,0\}$, which are suitable real or circle bundles over $M_G=\GG_{1,1}(\CC_1^{n+1})$. In Theorem \ref{thm-cv-CHn}, we obtain that, indeed,  these twistor spaces can be regarded as our families of parallel curves of constant curvature, lying inside totally geodesic $\mathbb{CH}^1\subset \CHn$. 
\par 
On the other hand, we recall from \cite{CK} the construction of a Gauss map for real hypersurfaces $M^{2n-1}$ in complex hyperbolic space $\CHn$ to $\GG_{1,1}(\CC_1^{n+1})$. In particular, when $M$ is Hopf, the rank of its Gauss map is $2n-2$, having nice properties related to  the para-quaternionic K\"ahler structure of $\GG_{1,1}(\CC_1^{n+1})$. Moreover, $\GG_{1,1}(\CC_1^{n+1})$ can be identified with the set of totally geodesic, complex hyperbolic lines $\mathbb{CH}^1$ in $\CHn$. Next, for each Hopf hypersurface $M^{2n-1}$ in $\CHn$, it is possible to obtain a horizontal submanifold $\Sigma^{2n-2}$ with respect to one of the twistor bundles $\pi_Z^s:M_Z^s\to M_G$ $(s\in\{+,-,0\})$, cf. \cite{K2,K3}. Theorem \ref{th-construction} is somehow a converse of this result. Indeed, starting from a $2n-2$ submanifold $\Sigma$ of  a twistor space $M_Z^s$, which is horizontal with respect to the twistor bundle $\pi_Z^s:M_Z^s\to M_G$ $(s\in\{+,-,0\})$, we construct a \textit{parallel} family of Hopf real hypersurfaces in $\CHn$, satisfying that $\vert\mu\vert>2$ for $s=+$, $\vert\mu\vert<2$ for $s=-$, and $\vert\mu\vert=2$ for $s=0$. 
\par 
In Section \ref{Examples}, we show how to recover the classical examples of Hopf real hypersurfaces in $\CHn$ from our point of view. We remark that, as far as we know, the only known example with $\vert\mu\vert=2$ is the horosphere. A remarkable, intriguing and long-standing problem since \cite{Mo} is the existence of other Hopf real hypersurfaces in $\CHn$ with $\mu=2$. We solve it in Section \ref{newexamples}, where we develop a method to construct plenty of such examples, summarized in Theorem \ref{Axi2xi}. For $n=2$, we also obtain a characterization of the horosphere in Theorem \ref{horo}.

\section{Real hypersurfaces in $\CHn$}
\label{RH}

We begin by  recalling a realization of the complex hyperbolic space $\mathbb{CH}^n$, 
according to \cite{KN}. In $\CC^{n+1}$ with  standard basis $\{e_0,e_1,\cdots,e_n\}$,  
we consider the Hermitian form 
\begin{equation}
((\zz,\ww))=-z_0\bar{w}_0+\sum_{k=1}^n z_k\bar{w}_k,
\label{Herm} 
\end{equation}
where $^t\zz=(z_0,z_1,\cdots,z_n), ^t\ww=(w_0,w_1,\cdots,w_n)\in \CC^{n+1}$. The corresponding scalar product is $\langle\zz,\ww\rangle = \mathrm{Re}((\zz,\ww))$. 

Let
\begin{align*}
 U(1,n) &=\{A\in GL(n+1,\mathbb{C})|\
 ((A\zz,A\ww))=((\zz,\ww)),\ \zz,\ww\in\mathbb{C}^{n+1}\}
\\
&=\{A\in GL(n+1,\mathbb{C})|\ A^*SA=S\},
\end{align*}
where
\begin{equation}
S=
\begin{pmatrix}
 -1 & 0 \\ 0 & E_n
\end{pmatrix}.
 \label{S}
\end{equation}
Then \eqref{Herm} and \eqref{S} imply $((\zz,\ww))=\,^t\zz S \bar{\ww}$.
If $u_0,u_1,\cdots,u_n$ are the column vectors of 
$A=(u_0,u_1,\cdots,u_n)\in U(1,n)$, we have  
\[
((u_0,u_0))=-1,\quad ((u_j,u_j))=1 \ (j=1,\cdots,n),\quad
((u_j,u_k))=0\ (j\not=k).
\]  
It is known that $U(1,n)$ is connected, and its Lie algebra is
\begin{align*}
 \mathfrak{u}(1,n) &=
\{X\in \mathfrak{gl}(n+1,\mathbb{C})|\ X^*S+SX=O\}
\\
&=\left\{
\begin{pmatrix}
 i\lambda & \mathbf{z}^* \\ \zz & B
\end{pmatrix}
|\ \lambda\in\RR,\ \zz\in\mathbb{C}^n,\ B\in \mathfrak{u}(n)
\right\}.
\end{align*}
Consider the real hypersurface \[H_1^{2n+1}=\{w\in \CC^{n+1} : ((w,w))=-1\},\] which is the \textit{anti de Sitter space} of constant sectional curvature $-1$. The group $U(1,n)$ acts transitively on it by isometries. The group $S^1=\{e^{i\theta}\}$ acts freely on $H_1^{2n+1}$ by $\ww\mapsto e^{i\theta}\ww$. The base manifold of the principal fiber bundle 
\[\pi_{-}: H_1^{2n+1}\to \CHn\] with group $S^1$ is the complex hyperbolic space. This is the well-known \textit{Hopf fibration}.  For $\ww\in H_1^{2n+1}$, the tangent space is represented by 
\[
T_{\ww}H_1^{2n+1}=\{\zz\in \CC^{n+1}|\ \langle \zz,\ww\rangle =0\}. 
\]
In particular, $i\ww\in T_\ww H_1^{2n+1}$. Let $T'_\ww$ be the subspace of 
$T_\ww H_1^{2n+1}$ defined by $T'_\ww=\{\zz\in \CC^{n+1}|\ ((\zz,\ww))=0\}.$
The restriction of $\langle,\rangle=\mathrm{Re}((\ ,\ ))$ to $T'_\ww$ is positive-definite. Since $T'_\ww$ is horizontal with respect to $\pi_{-}$, we have a connection on the fiber bundle. Then  $(\pi_{-})_*$ induces a linear isomorphism from $T'_\ww$ to $T_{\pi_{-}(\ww)}\CHn$ and the complex structure $\zz\mapsto i\zz$ on $T'_\ww$, $\ww\in H_1^{2n+1}$, is compatible with the action of $S^1$, so that it induces an almost complex structure $J'$ on $\CHn$ such that  $(\pi_{-})_*i=J'(\pi_{-})_*$.  Given $X\in T_p\CHn$, its \textit{horizontal lift} will be denoted by $X'\in T'_\ww$. Given $X\in T_zH_1^{2n+1}$, its \textit{horizontal part} is the orthogonal projection of $X$ into $T'_\ww$, and we will denote it by $\mathcal{H}X=X+\mathrm{Re}((X,i\ww))i\ww$. 
Next, by abuse of notation, we define on $\CHn$ a Riemannian metric by 
\[
\langle X,Y\rangle=\operatorname{Re} ((X',Y')),
\] 
where $X',Y'\in T'_\ww$ are horizontal lifts of $X,Y\in T_p\CHn$, $\pi_{-}(\ww)=p$. 
The metric $\langle\ ,\ \rangle$ is Hermitian with respect to $J'$ and $\mathbb{CH}^n$ has constant holomorphic sectional curvature $-4$.
\par
Next, we recall some facts about real hypersurfaces in complex space forms.
Let $\widetilde{M}^n(c)$ be a space of constant holomorphic sectional 
curvature $4c$ with real dimension $2n$ and Levi-Civita connection 
$\tilde{\nabla}$.  For a real hypersurface $M^{2n-1}$ in $\widetilde{M}$, 
the Levi-Civita connection $\nabla$ of the induced metric and 
the shape operator $A$ are characterized by
\[\tilde{\nabla}_X Y=\nabla_X Y+\langle AX,Y\rangle N, \quad 
\tilde{\nabla}_X N=-AX, 
\]
where $N$ is a local choice of unit normal. Let $J:T\widetilde{M}\rightarrow \wtM$ be
the complex structure with properties $J^2=-1$, $\widetilde{\nabla}J=0$, 
and $\langle JX,JY\rangle=\langle X,Y\rangle$.
\par
We define the \textit{structure vector} of $M$ by $\xi=-JN$.
Clearly $\xi\in TM$ and $|\xi|=1$. Define a skew-symmetric $(1,1)$-tensor $\phi$ from the tangent 
projection of $J$ by
\[ JX=\phi X+\eta(X)N, \]
where $\eta$ is the $1$-form on $M$ defined by $\eta(X)=\langle \xi,X\rangle$.
Therefore, 
\[
\phi^2X=-X+\eta(X)\xi,\quad \langle\phi X,\phi Y\rangle=\langle X,Y\rangle 
-\eta(X)\eta(Y),\quad \phi\xi=0.
\]
$(\phi,\xi,\eta,\langle\ ,\ \rangle)$ determines an almost-contact 
metric structure on $M$. Since $\tilde{\nabla}J=0$, 
\[
(\nabla_X \phi)Y=\eta(Y)AX-\langle AX,Y\rangle\xi, \quad
\nabla_X\xi=\phi AX.
\]
\par
If $\xi$ is a principal vector, i.e., $A\xi=\mu\xi$ holds for the 
shape operator of $M$ in $\wtM$, $M$ is called a \textit{Hopf hypersurface}
and $\mu$ is called \textit{Hopf curvature}. 
The following fact is proved by Y.~Maeda \cite{Mae} and Ki-Suh \cite{KS} 
(cf. \cite{NR}, Corollary 2.3 (i)).
\begin{proposition} \label{Prop-Hopf}
Let $M^{2n-1}$ $(n\ge 2)$ be a connected Hopf hypersurface in a complex space
form of constant holomorphic sectional curvature $4c\not=0$, 
and let $\mu$ be its Hopf curvature. Then $\mu$ must be a constant. 
Furthermore if $X\perp\xi$ and $AX=\lambda X$, then 
\begin{equation}
\left(\lambda-\frac{\mu}{2}\right)A\phi X=
\left(\frac{\lambda\mu}{2}+c\right)\phi X.
\label{Hopf-pc} 
\end{equation}
Hence 
\begin{equation}
 A\phi X=\left(\frac{\lambda\mu+2c}{2\lambda-\mu}\right)\phi X\quad
\left(\text{provided}\quad \lambda-\frac{\mu}{2}\not=0\right).
\label{Hopf-pc2}
\end{equation}
\end{proposition}
We note that $\mu$ is not constant for Hopf hypersurfaces in $\CC^{n+1}$ 
in general (\cite{NR}, Theorem 2.1). In fact, in \eqref{Hopf-pc} if $\lambda=\mu/2$, then 
$\mu^2+4c=0$ and $c<0$. 
\begin{corollary}
Let $M^{2n-1}$ $(n\ge 2)$ be a Hopf hypersurface with Hopf curvature
$\mu$ in a complex space form of constant holomorphic sectional curvature $4c\not=0$.
If there exists a (locally defined) vector field $X\not=0$ 
orthogonal to $\xi$ with $AX=(\mu/2) X$, then $c<0$, $|\mu|=2\sqrt{-c}$ and $|\lambda|=\sqrt{-c}$.
In particular, when either $c>0$ or $c<0$ with $|\mu|\not=2\sqrt{-c}$,
there exists no tangent vector field $X\not=0$ orthogonal to $\xi$ 
satisfying $AX=(\mu/2) X$. 
\end{corollary}
\begin{corollary}
Let $M^{2n-1}$ $(n\ge 2)$ be a Hopf hypersurface with Hopf curvature $\mu$ in a complex hyperbolic space $\CHn(-4)$. Take $X$ a principal vector such that $X\perp\xi$,  $AX=\lambda X$ and $A\phi X=\lambda^*\phi X$. Then:
\begin{enumerate}
\item In the case $|\mu|>2$, we put $\mu=2\coth 2r$ for some $r\not=0$.
\begin{enumerate}
\item If $|\lambda|>1$, then $|\lambda^*|>1$, $\lambda=\coth(r+\theta)$, $\lambda^*=\coth(r-\theta)$ for some $\theta\in\RR$. 

\item If $|\lambda|<1$, then $|\lambda^*|<1$, $\lambda=\tanh(r+\theta)$, $\lambda^*=\tanh(r-\theta)$, for some $\theta\in\RR$.

\item If $\lambda=\pm 1$, then $\lambda^*=-\lambda$. 

\item Given the eigenvalues $\lambda_1,\cdots,\lambda_{2n-2}$ of 
$A|_{\{\xi\}^\perp}$, the numbers $\#\{i|\ |\lambda_i|>1\}$ and $\#\{i|\ |\lambda_i|<1\}$ 
are even.
\end{enumerate}
\item In the case $|\mu|<2$, we put $\mu=2\tanh 2r$ for some $r\in\RR$. 
\begin{enumerate}
\item If $|\lambda|>1$, then $|\lambda^*|<1$, $\lambda=\coth(r+\theta)$, $\lambda^*=\tanh(r-\theta)$ for some $\theta\in\RR$ (and vice versa). 
\item If $\lambda=\pm 1$, then $\lambda^*=-\lambda$. 
\item Given the eigenvalues $\lambda_1,\cdots,\lambda_{2n-2}$ of 
$A|_{\{\xi\}^\perp}$, the numbers  $\#\{i|\  |\lambda_i|>1\}$ and $\#\{i |\: |\lambda_i|<1\}$ are equal.
 \end{enumerate}
\item In the case $\mu=\pm 2$, if $\lambda\neq \mu/2$, then $\lambda^*=\pm 1$, that is, 
$A\phi X=\pm \phi X$. 
\end{enumerate}
\end{corollary}

A fundamental result for Hopf hypersurfaces in complex projective space 
$\CPn$ was proved by Cecil-Ryan in \cite{CR}. They considered a Hopf real hypersurface $M$ with $\mu=2 \cot 2r$, for some $0<r<\pi/2$, and such that its focal map has constant rank $q$. Roughly speaking, they proved that  every point of $M$ admits an open neighbourhood $U$ such that $U$ lies on a tube of radius $r$ over a complex submanifold of $\CPn$. The corresponding result in $\CHn$ for \textit{large} Hopf curvatures was proved by Montiel, \cite{Mo}:
\begin{theorem}	\label{th-Montiel}
Let $M$ be an orientable Hopf hypersurface of $\CHn$ 
and we assume that its focal map $\phi_{r}$ $(r>0)$ has constant rank $q$ on M. Then  if
$\mu=2\coth 2r>2$ (for some $r>0$), for each $x_{0}\in M$ there exists 
an open neighbourhood $U$ of $x_{0}$ such that $\phi_{r}U$ is a 
$(q/2)$-dimensional complex submanifold embedded in $\CHn$. 
Moreover $U$ lies in a tube of radius $r$ over $\phi_{r}U$.
\end{theorem}

On the other hand, with respect to Hopf hypersurfaces with 
\textit{small} Hopf curvature in $\CHn$, Ivey \cite{I} and 
Ivey-Ryan \cite{IR} proved that a Hopf hypersurface with
$|\mu|<2$ in $\CHn$ may be constructed from an arbitrary pair of 
Legendrian submanifolds in $S^{2n-1}$.

\section{Para-quaternionic K\"ahler manifolds}
\label{PQK} 

We review basic definitions and facts on 
a para-quaternionic K\"ahler manifold 
(cf. \cite{Mar}, \cite{Ts}). First, we recall real Clifford algebras (cf. \cite{H, Mar}) and 
(para-)quaternionic structures (cf. \cite{Mar}).
Let $(V=\RR(p,q),\langle\ ,\ \rangle)$ be a real symmetric inner product space 
of signature $(p,q)$. The Clifford algebra $C(p,q)$ is the quotient 
$\otimes V/I(V)$, where $I(V)$ is the two-sided ideal in $\otimes V$ 
generated by all elements $x\otimes x+\langle x,x\rangle$ with $x\in V$. The Clifford algebras with $p+q=1$ or $2$ are given as:
\begin{enumerate}
\item $\CC=C(0,1)$, \textit{complex numbers}: 
$z=x+iy$ $(x,y\in\RR)$, $i^2=-1$, $|z|^2=x^2+y^2$, there are no zero divisors.
\item $\widetilde{\CC}=C(1,0)$, \textit{split-complex numbers}: 
$z=x+jy$ $(x,y\in\RR)$, $j^2=1$, $|z|^2=x^2-y^2$, there exist zero divisors.
\item $\HH=C(0,2)$, \textit{quaternions}:
$q=q_0+iq_1+jq_2+kq_3$ $(q_0,q_1,q_2,q_3\in\RR)$, $i^2=j^2=k^2=ijk=-1$, 
$|q|^2=q\bar{q}=q_0^2+q_1^2+q_2^2+q_3^2$,
there are no zero divisors,
$\HH\cong \CC^2$, $q=z_1(q)+jz_2(q)$.
\item $\widetilde{\HH}=C(2,0)=C(1,1)$, \textit{split-quaternions}:
$q=q_0+iq_1+jq_2+kq_3$ $(q_0,q_1,q_2,q_3\in\RR)$, $i^2=-1$, $j^2=k^2=ijk=1$, 
$|q|^2=q_0^2+q_1^2-q_2^2-q_3^2$,
there exist zero divisors,
$\widetilde{\HH}\cong \CC^2$, $q=z_1(q)+jz_2(q)$.
\end{enumerate}
The following natural correspondences are well-known:
(i) the split-complex numbers $\widetilde{\CC}$ and the Minkowski plane $\RR^2_1$, 
and (ii) the split-quaternions $\widetilde{\HH}$ and the 
Euclidean 4-space $\RR^4_2$ of signature $(2,2)$. 
\par 
Next let $(\wtM^{4m},\tilde{g},\tilde{Q})$, $m\ge 2$,  be a \textit{para-quaternionic}
K\"ahler manifold with the para-quaternionic K\"ahler structure
$(\tilde{g},\tilde{Q})$, that is to say, $\tilde{g}$ is a pseudo-Riemannian metric 
of neutral signature $(2m,2m)$ on $\wtM$ and $\tilde{Q}$ is a rank 3 subbundle of 
$\text{End} T\wtM$ satisfying	 the following conditions:
\begin{enumerate}
\item
For each $p\in\wtM$ , there exist a neighbourhood $U$ of $p$ and a (local) frame field  $\{\tilde{I}_1,\tilde{I}_2,\tilde{I}_3\}$ of $\tilde{Q}$ defined on $U$ satisfying
\begin{gather*}
\tilde{I}_1^2=-1,\quad \tilde{I}_2^2=\tilde{I}_3^2=1,\quad 
\tilde{I}_1\tilde{I}_2=-\tilde{I}_2\tilde{I}_1=-\tilde{I}_3, \\
\tilde{I}_2\tilde{I}_3=-\tilde{I}_3\tilde{I}_2=\tilde{I}_1,\quad 
\tilde{I}_3\tilde{I}_1=-\tilde{I}_1\tilde{I}_3=-\tilde{I}_2.
\end{gather*}
\item For each  $L\in \tilde{Q}_p$, 
$\tilde{g}_p(LX,Y)+\tilde{g}_p(X,LY)=0$ for $X,Y\in T_p\wtM$, $p\in\wtM$.
\item The vector bundle $\tilde{Q}$ is parallel in $\operatorname{End} T\wtM$ 
with respect to the Levi-Civita connection 
$\widetilde{\nabla}$ associated with $\tilde{g}$.
\end{enumerate}
We note that $\tilde{Q}_p$ is naturally identified with Lie algebra 
$\su(1,1)$ and isometric to Minkowski $3$-space $\RR_1^3$ with respect to 
its Killing form:
$$
\tilde{Q}_p=\{\tilde{I}=aI_1+bI_2+cI_3|\ a,b,c\in\RR\}
\cong \su(1,1) \cong \RR_1^3,
$$ 
and
\begin{align}
 (\mathbf{S}_+)_p &:=\{\tilde{I}\in \tilde{Q}_p|\ \tilde{I}^2=1\}\cong S_1^2:\ 
(\text{de-Sitter $2$-space}), 
\label{S+}\\
 (\mathbf{S}_-)_p &:=\{\tilde{I}\in \tilde{Q}_p|\ \tilde{I}^2=-1,\ a>0\}\cong H^2:\ 
(\text{hyperbolic $2$-space}), 
\label{S-}\\
 (\mathbf{S}_0)_p &:=\{\tilde{I}\in \tilde{Q}_p|\ \tilde{I}^2=0,\
 \tilde{I}\not=0\}
\cong L^2-\{0\}:\ 
(\text{light cone}).
\label{S0}
\end{align}

We introduce now three \textit{twistor spaces} $\mathcal{Z}_s$ ($s=-,+,0$)
by using the para-quaternionic K\"ahler structure, such that $\mathcal{Z}_-$ (resp. $\mathcal{Z}_+$ 
and $\mathcal{Z}_0$) is a hyperbolic plane bundle (resp. de-Sitter surface bundle and 
light-cone bundle) over  $\bar{M}$. Each fiber 
of $\mathcal{Z}_s\rightarrow\bar{M}$ is identified with 
the set of almost complex (resp. almost product and nilpotent complex) 
structure $I$ with $I^2=-1$ (resp. $I^2=1$ and $I^2=0$) 
of the tangent space of $\bar{M}$, 
i.e., $I=a_1I_1+a_2I_2+a_3I_3$ with $(a_1,a_2,a_3)\in S^2$ and  
$-a_1^2+a_2^2+a_3^2=-1$ (resp. $-a_1^2+a_2^2+a_3^2=1$ and 
$-a_1^2+a_2^2+a_3^2=0$).
The cases $\mathcal{Z}_-$ and $\mathcal{Z}_+$ were studied by Alekseevsky-Cort\'es in \cite{AC}.

\section{Twistor spaces of indefinite complex $2$-plane Grassmannian}
\label{Twist}

We are going to identify the three previously said twistor spaces on 
$M_G=\GG_{1,1}(\CC_1^{n+1})$  with respect to its 
para-quaternionic K\"ahler structure. 

Given the usual scalar product in $\CC_1^{n+1}$ as in Section \ref{RH}, we define a scalar product (cf. (13) in \cite{CK}) 
on $\CC_1^{n+1}\times \CC_1^{n+1}$, $n\ge 2$, by 
\begin{equation}
 \langle(X_-,X_+),(Y_-,Y_+)\rangle:=
-\langle X_-,Y_-\rangle+\langle X_+,Y_+\rangle,
\label{sc-prd}
\end{equation}
where $(X_-,X_+),(Y_-,Y_+)\in \CC_1^{n+1}\times \CC_1^{n+1}$, 
whose signature is $(2n+2,2n+2)$. Let 
\begin{gather}
 M_V:=V_{1,1}(\CC_1^{n+1})=\{(u_-,u_+)\in \CC_1^{n+1}\times\CC_1^{n+1}|
\label{Stiefel}\\
\langle u_-,u_-\rangle=-1,\langle u_+,u_+\rangle=1,
\langle u_-,u_+\rangle=\langle u_-,iu_+\rangle=0\}
\notag
\end{gather}
be the \textit{indefinite complex Stiefel manifold of 
orthonormal time-like and space-like vectors in}
$\CC_1^{n+1}$.
\par
The group $U(1,n)$ acts on $M_V$ as 
$g\cdot(u_-,u_+):=(gu_-,gu_+)$ for $g\in U(1,n)$.
The action is transitive and the isotropy group 
at $(e_1,e_2)\in M_V$, $e_1=^t\!\!(1,0,\cdots,0)$ and 
$e_2=^t\!\!(0,1,0,\cdots,0)$ is $\{E_2\}\times U(n-1)$, 
where $E_2$ is the $2\times 2$ identity matrix. Hence,  
$M_V$ is represented by $U(1,n)/(\{E_2\}\times U(n-1))$ as a homogeneous space.
By \eqref{Stiefel}, the tangent space is
\begin{gather}
 T_{(u_-,u_+)}M_V=
(\{u_-,u_+\}^\perp\times \{u_-,u_+\}^\perp) 
\label{TV}\\
\oplus \RR(iu_-,iu_+)\oplus\RR(iu_-,-iu_+)\oplus\RR(u_+,u_-)
\oplus \RR(iu_+,-iu_-),
\notag
\end{gather}
where $\{u_-,u_+\}^\perp$ is the orthogonal complement of 
complex $2$-plane spanned by $u_-$ and $u_+$ in $\CC_1^{n+1}$, and 
$\RR(*,**)$ is the real line spanned by $(*,**)$ in $\CC_1^{n+1}$.
The signature of the scalar produce induced by \eqref{sc-prd} 
on $T_{(u_-,u_+)}M_V$ is $(2n,2n)$.
\par
The indefinite complex $2$-plane Grassmannian 
$$
M_G:=\GG_{1,1}(\CC_1^{n+1})
=\{\operatorname{span}_\CC\{u_-,u_+\}|\ (u_-,u_+)\in M_V\}
$$ 
is represented as a homogeneous space $U(1,n)/(U(1,1)\times U(n-1))$. 
Also, $M_G$ is obtained as the quotient space of $M_V$ by the action of $U(1,1)$:
\begin{equation}
U(1,1)\ni h,\quad (u_-,u_+)\mapsto (u_-,u_+)h,
\label{act11} 
\end{equation}
with projection $\pi^{VG}:M_V\rightarrow M_G$. Thus, the tangent space $T_{\pi^{VG}(u_-,u_+)}M_G$ is  
identified with the horizontal subspace 
\begin{equation}
T'_{(u_-,u_+)}:=\{u_-,u_+\}^\perp\times \{u_-,u_+\}^\perp 
\label{TG} 
\end{equation}
of $T_{(u_-,u_+)}M_V$ through the differential map of $\pi^{VG}$.
\par
The subspace $T'_{(u_-,u_+)}$ admits a complex structure $(X_-,X_+)\mapsto (iX_-,iX_+)$.
By transferring the complex structure $i$ and the inner product $\tilde{g}$ on 
$T'_{(u_-,u_+)}$ by $\tilde{\pi}_*$, we get the usual complex structure $J$ 
and the pseudo-K\"ahler metric $g$ on $\GG_{1,1}(\CC_1^{n+1})$, 
respectively. Note that $\GG_2(\CC^{n+1})$ is a complex $2(n-1)$-dimensional 
Hermitian symmetric space.  
\par 
Given $U$ an open subset of $M_G$, let $s:U\rightarrow s(U)\subset M_V$, $p\mapsto (u_-(p),u_+(p))$ 
be a local section with respect to $\pi^{VG}$. We define a local basis of a para-quaternionic K\"ahler structure 
on $M_G$ as follows:
\begin{gather}\label{paraK}
 I_1:(X_-,X_+)\mapsto (iX_-,-iX_+)
=(X_-,X_+)\begin{pmatrix} i & 0 \\ 0 & -i
	  \end{pmatrix},\notag \\
 I_2:(X_-,X_+)\mapsto (X_+,X_-)
=(X_-,X_+) \begin{pmatrix} 0 & 1 \\ 1 & 0
	   \end{pmatrix},
\label{para-QK-MG} \\
 I_3:(X_-,X_+)\mapsto (iX_+,-iX_-)
=(X_-,X_+)\begin{pmatrix} 0 & -i \\ i & 0
 \end{pmatrix},
\notag
\\
(X_-,X_+)\in T'_{(u_-(p),u_+(p))}
\cong T_{\pi^{VG}(u_-(p),u_+(p))}M_G, \notag
\end{gather}
through the identification  \eqref{TG}.
Then  for $I=aI_1+bI_2+cI_3$ with $a,b,c\in\RR$,  it holds
\begin{gather*}
 I\in (\mathbf{S}_+)_{\pi^{VG}(u_-(p),u_+(p))}\ \Leftrightarrow \ 
-a^2+b^2+c^2=1,\\
 I\in (\mathbf{S}_-)_{\pi^{VG}(u_-(p),u_+(p))}\ \Leftrightarrow \ 
-a^2+b^2+c^2=-1,\ a>0\\
 I\in (\mathbf{S}_0)_{\pi^{VG}(u_-(p),u_+(p))}\ \Leftrightarrow \ 
-a^2+b^2+c^2=0,\ (a,b,c)\not=(0,0,0),
\end{gather*}
where $\mathbf{S}_+,\mathbf{S}_-$ and $\mathbf{S}_0$ are given by 
\eqref{S+}, \eqref{S-} and \eqref{S0}, respectively.
\par
The manifolds 
\begin{equation}
 M_Z^s:=\{I\in (\mathbf{S}_s)_{\pi^{VG}(u_-,u_+)}|\ (u_-,u_+)\in M_V\}
\quad
(s=-,+,0)
\label{def-twst}
\end{equation}
are the \textit{twistor spaces of $M_G$ with respect to the para-quaternionic 
K\"ahler structure}. We define an action of $U(1,n)$ on $M_Z^s$ as follows:
\begin{gather}
\text{For}\ \  g\in U(1,n), \ \ 
(X_-,X_+)\in T'_{g(u_-,u_+)}\ \ 
\text{and}\ \  
I\in (\mathbf{S}_s)_{\pi^{VG}(u_-,u_+)},
\notag\\
(g I)(X_-,X_+):=g(I(g^{-1}X_-,g^{-1}X_+)).
\label{action}
\end{gather}
This action is transitive. Since $I$ is a linear endomorphism of a tangent space of 
$M_G$, each isotropy group $H$ of the action of $U(1,n)$ 
on $M_Z^s$ is contained in $U(1,1)\times U(n-1)$.
Also, at $(e_1,e_2)\in M_V$, each element of $\{E_2\}\times U(n-1)$ 
fixes $I$ for  all $(X_-,X_+)\in T'_{(e_1,e_2)}$ in \eqref{action}, 
so $\{E_2\}\times U(n-1)$ lies in $H$. On the other hand, any matrix  in $U(1,1)$ is 
written as $(e^{i\theta}E_2) g$, $(\theta\in\RR,\ g\in SU(1,1))$, so that 
$\{e^{i\theta}E_2\}\times \{E_{n-1}\}$ lies in $H$.
Thus, we need to determine those elements $g\in SU(1,1)$ such that 
$\{g\}\times \{E_{n-1}\}$ lies in the isotropy group $H$.
\par
Given $(u_-,u_+)\in M_V$, we put
\begin{gather}
g:=\begin{pmatrix} \alpha & \bar{\beta} \\ \beta & \bar{\alpha}
	   \end{pmatrix}\in SU(1,1),\quad
|\alpha|^2-|\beta|^2=1,
\notag
\\
 (\tilde{u}_-,\tilde{u}_+)
:=(u_-,u_+)g
=(\alpha u_-+\beta u_+,\bar{\beta}u_-+\bar{\alpha}u_+).
\end{gather}
This induces the map
$$
(\pi^{VG})^{-1}(\pi^{VG}(u_-,u_+))\rightarrow 
(\pi^{VG})^{-1}(\pi^{VG}(u_-,u_+)),\quad
(u_-,u_+)\mapsto (\tilde{u}_-,\tilde{u}_+),
$$
whose differential is given by
$$
T'_{(u_-,u_+)}\ni(X_-,X_+)
\mapsto
(\tilde{X}_-,\tilde{X}_+):=(X_-,X_+)g
\in T'_{(\tilde{u}_-,\tilde{u}_+)}.
$$
If we consider another local basis of the para-quaternionic K\"ahler structure 
on $M_G$, 
\begin{gather}
 \tilde{I}_1:(\tilde{X}_-,\tilde{X}_+)\mapsto (i\tilde{X}_-,-i\tilde{X}_+)
=(\tilde{X}_-,\tilde{X}_+)\begin{pmatrix} i & 0 \\ 0 & -i
	  \end{pmatrix},\notag \\
 \tilde{I}_2:(\tilde{X}_-,\tilde{X}_+)\mapsto (\tilde{X}_+,\tilde{X}_-)
=(\tilde{X}_-,\tilde{X}_+) \begin{pmatrix} 0 & 1 \\ 1 & 0
	   \end{pmatrix},
\label{para-QK-MG2} \\
 I_3:(\tilde{X}_-,\tilde{X}_+)\mapsto (i\tilde{X}_+,-i\tilde{X}_-)
=(\tilde{X}_-,\tilde{X}_+)\begin{pmatrix} 0 & -i \\ i & 0
 \end{pmatrix},
\notag
\\
(\tilde{X}_-,\tilde{X}_+)\in \{u_-(p),u_+(p)\}^\perp\times \{u_-(p),u_+(p)\}^\perp 
\cong T_{\pi^{VG}(u_-(p),u_+(p))}M_G,
\notag
\end{gather}
then  we have
\begin{align*}
 \tilde{I}_1
&=gI_1g^{-1}
=
(|\alpha|^2+|\beta|^2)I_1
+2\operatorname{Re}(\bar{\alpha}\beta)I_2
+2\operatorname{Im}(\bar{\alpha}\beta)I_3, \\
 \tilde{I}_2
&=gI_2g^{-1}
=
-2\operatorname{Im}(\alpha\beta)I_1
+2\operatorname{Re}(\alpha^2-\bar{\beta}^2)I_2
+2\operatorname{Im}(\alpha^2-\bar{\beta}^2)I_3,
\\
 \tilde{I}_3
&=gI_3g^{-1}
=
2\operatorname{Re}(\alpha\beta)I_1
+2\operatorname{Re}(\bar{\alpha}^2+\beta^2)I_2
+2\operatorname{Im}(\bar{\alpha}^2+\beta^2)I_3.
\end{align*}
Therefore,  $\tilde{I}_1=I_1$ $\Leftrightarrow$ 
$|\alpha|^2+|\beta|^2=1$, $\bar{\alpha}\beta=0$.
Then $|\alpha|^2-|\beta|^2=1$ implies 
$|\alpha|^2=1$, $\beta=0$ and 
$$
\tilde{I}_1=I_1 \Leftrightarrow 
\begin{pmatrix}
 \alpha & \bar{\beta} \\ \beta & \bar{\alpha}
\end{pmatrix} 
=
\begin{pmatrix}
 e^{it} & 0 \\ 0 & e^{-it}
\end{pmatrix}
\quad (t\in\RR).
$$
Also, $\tilde{I}_2=I_2$ $\Leftrightarrow$ 
$\operatorname{Im}(\alpha\beta)=0$, $\alpha^2-\bar{\beta}^2=1$.
Then $|\alpha|^2-|\beta|^2=1$ implies 
$\alpha,\beta\in\RR$ and
$$
\tilde{I}_2=I_2 \Leftrightarrow 
\begin{pmatrix}
 \alpha & \bar{\beta} \\ \beta & \bar{\alpha}
\end{pmatrix} 
=
\pm
\begin{pmatrix}
 \cosh t & \sinh t \\ \sinh t & \cosh t
\end{pmatrix}
\quad (t\in\RR).
$$
Finally, $\tilde{I}_1+\tilde{I}_2=I_1+I_2$ $\Leftrightarrow$ 
$|\alpha|^2+|\beta|^2-2\operatorname{Im}(\alpha\beta)=1$, 
$2i\bar{\alpha}\beta+\bar{\alpha}^2-\beta^2=1$.
The second equation is 
$(\bar{\alpha}+i\beta)^2=1$ and so $\bar{\alpha}+i\beta=\pm 1$.
Also $|\alpha|^2-|\beta|^2=1$ implies 
$|\alpha|^2=1+\operatorname{Im}(\alpha\beta)$ and 
$|\beta|^2=\operatorname{Im}(\alpha\beta)$.
Therefore, $\operatorname{Re}\alpha=\pm 1$ and 
$\operatorname{Im}\beta=0$. If we put $\beta=t\in\RR$, then  $\alpha=it\pm 1$. 
We summarize all these computations in the following result.
\begin{proposition}
Each twistor space
\begin{gather*}
 M_Z^s:=\{I\in (\mathbf{S}_s)_{\pi^{VG}(u_-,u_+)}|\ (u_-,u_+)\in M_V\} 
\quad (s=-,+,0)
\end{gather*}
of the indefinite complex $2$-plane Grassmannian $M_G$ with 
respect to the para-quaternionic K\"ahler structure $Q$ 
is represented as a homogeneous space $U(n+1)/(H_s\times U(n-1))$, $n\ge 2$,
where respectively 
\begin{gather*}
H_+=\left\{e^{i\theta}
\begin{pmatrix}
 \cosh t & \sinh t \\ \sinh t & \cosh t
\end{pmatrix}
\biggl|\biggr.\ \theta,t\in\RR\right\}, \\
H_{-}=T^2=S^1\times S^1, \quad  
H_0=\left\{e^{i\theta}
\begin{pmatrix}
 1+it & t \\ t & 1-it
\end{pmatrix}
\biggl|\biggr.\ \theta,t\in\RR\right\}.
\end{gather*}
\end{proposition}
\par
We introduce the following action of $S^1$ on $M_V$ as 
\begin{equation}
S^1=\RR/2\pi\ZZ\ni\theta,\quad 
 (u_-,u_+)\mapsto(e^{i\theta}u_-,e^{i\theta}u_+). 
\label{act1}
\end{equation}
Let $M_S$ be the quotient space of the action \eqref{act1} 
with  projection $\pi^{VS}:M_V\rightarrow M_S$.
The tangent space $T_{\pi^{VS}(u_-,u_+)}M_S$ is 
identified with the subspace 
\begin{equation}
(\{u_-,u_+\}^\perp\times \{u_-,u_+\}^\perp )
\oplus\RR(iu_-,-iu_+)\oplus\RR(u_+,u_-)
\oplus \RR(iu_+,-iu_-)
\label{TS} 
\end{equation}
of $T_{(u_-,u_+)}M_V$ through the differential map of $\pi^{VS}$, 
and $\operatorname{ker}\big(d\pi^{VS}\big)=\RR(iu_-,iu_+).$
\par
We consider $3$ actions of $S^1$, $SO_+(1,1)$ and $\RR$ on 
$M_V$, which commute with the action \eqref{act1} of $S^1$ as:
\begin{gather}
 S^1=\RR/2\pi\ZZ\ni t,\quad (u_-,u_+)\mapsto
 (e^{it}u_-,e^{-it}u_+),
\label{act+}\\
SO(1,1)\ni \begin{pmatrix} \cosh t & \sinh t \\
	  \sinh t & \cosh t
	 \end{pmatrix}, \quad
(u_-,u_+)\mapsto
(u_-,u_+)\begin{pmatrix} \cosh t & \sinh t \\
	  \sinh t & \cosh t
	 \end{pmatrix},
\label{act-}\\
\RR\ni t,\quad 
(u_-,u_+)\mapsto (u_-,u_+)
\begin{pmatrix}
 1+it & t \\ t & 1-it
\end{pmatrix}.
\label{act0}
\end{gather}
Then $S^1$, $SO_+(1,1)$ and $\RR$ also act on $M_S$ 
by \eqref{act+}, \eqref{act-} and \eqref{act0}, respectively.
We denote by $M^+_Z$, $M^-_Z$ and $M^0_Z$, the quotient 
space of $M_S$ by the actions  
\eqref{act+}, \eqref{act-} and \eqref{act0}, 
with projections $\pi_+^{SZ}:M_S\rightarrow M^+_Z$, 
$\pi_-^{SZ}:M_S\rightarrow M^-_Z$ and 
$\pi_0^{SZ}:M_S\rightarrow M^0_Z$, respectively.
We also have the projections $\pi_+^{Z}:M^+_Z\rightarrow M_G$, 
$\pi_-^{Z}:M^0_Z\rightarrow M_G$ and 
$\pi_0^{Z}:M^-Z\rightarrow M_G$, respectively, such that
the following diagram of fibrations is commutative:
\begin{equation}
\xymatrix{
 & & M^+_Z \ar[rd]_{\pi_+^Z} & 
\\
M_V \ar[r]_{\pi^{VS}} \ar@/^60pt/@{>}[rrr]^{\pi^{VG}} 
& M_S \ar[ru]^{\pi_+^{SZ}} 
\ar[r]_{\pi_0^{SZ}} \ar[rd]_{\pi_-^{SZ}}  & M^0_Z \ar[r]_{\pi_0^Z} & M_G. 
\\
 & & M^-_Z \ar[ru]_{\pi_-^Z} & 
}
\label{com-diag}
\end{equation} 
\par
Let $(u_-,u_+)\in M_V$ be a point and let $\pi^{VG}(u_-,u_+)\in M_G$ be 
an indefinite $2$-plane in $\CC_1^{n+1}$ spanned by $u_-$ and $u_+$.
Recall \eqref{para-QK-MG}.
\par
The vertical part of $d\pi_s^{SZ}$ is spanned by
$$
\begin{cases}
 (iu_-,-iu_+) & (s=+), \\
 (u_+,u_-) & (s=-), \\
 (iu_-+u_+,u_--iu_+) & (s=0),
\end{cases}
$$
and the
tangent spaces of $M^+_Z$, $M^-_Z$ and $M^0_Z$ are 
respectively described as follows:  
\begin{gather}
 T_{\pi^+_Z\circ\pi^{VS}(u_-,u_+)}M^+_Z
\cong 
(\{u_-,u_+\}^\perp\times \{u_-,u_+\}^\perp )
\oplus\RR(u_+,u_-)
\oplus \RR(iu_+,-iu_-),
\notag\\
 T_{\pi^-_Z\circ\pi^{VS}(u_-,u_+)}M^-_Z
\cong 
(\{u_-,u_+\}^\perp\times \{u_-,u_+\}^\perp )
\oplus\RR(iu_-,-iu_+)
\oplus \RR(iu_+,-iu_-),
\label{TZ}\\
 T_{\pi^0_Z\circ\pi^{VS}(u_-,u_+)}M^0_Z
\cong 
(\{u_-,u_+\}^\perp\times \{u_-,u_+\}^\perp )
\oplus\RR(iu_--u_+,-u_--iu_+)
\oplus \RR(iu_+,-iu_-).
\notag
\end{gather}
The vertical part of $d\pi_s^{Z}$ is spanned by
$$
\begin{cases}
 (u_+,u_-),\ (iu_+,-iu_-) & (s=+), \\
 (iu_-,-iu_+),\ (iu_+,-iu_-) & (s=-), \\
 (iu_--u_+,-u_--iu_+),\ (iu_+,-iu_-) & (s=0),
\end{cases}
$$
and the horizontal subspace of 
$ T_{\pi^s_Z\circ\pi^{VS}(u_-,u_+)}M^s_Z$ ($s=+,-,0$) is 
identified with $T'_{(u_-,u_+)}$ of \eqref{TG}. 
\par 
Our next target is to study when there are local lifts of an  immersion 
$\varphi_s:\Sigma\rightarrow M^s_Z$ to $\tilde{\varphi}_s:U\rightarrow M_V$. 
Let $\Sigma$ be a manifold and let 
$\varphi_s:\Sigma\rightarrow M^s_Z$ be an immersion.
Assume that for some point $p\in\Sigma$, there exists an open neighbourhood 
$U$ of $p$ in $\Sigma$, and a local lift 
$\tilde{\varphi}_s:U\rightarrow M_V$ 
$\tilde{\varphi}_s(p)=(u_-(p),u_+(p))$, of $\varphi_s$ 
satisfying $(\pi^{SZ}_s\circ\pi^{VS})\circ\tilde{\varphi}_s=\varphi_s$.
As $\CC_1^{n+1}$-valued functions on $U$, the differential of $u_-$ and 
$u_+$ are given by
\begin{equation}
 du_-=i\alpha_-u_-+\beta u_++w_-, \qquad   
 du_+=\overline{\beta}u_-+i\alpha_+u_++w_+,
\label{diff-u-+}
\end{equation}
where $\alpha_-$ and $\alpha_+$ (resp. $\beta$) are $\RR$-valued 
(resp. $\CC$-valued), and $w_-$, $w_+$ are 
$\{u_-,u_+\}^\perp$-valued $1$-forms on $U$. By \eqref{TV}, \eqref{TG} and \eqref{TZ}, $\varphi_s$ is 
\textit{horizontal} with respect to the \textit{twistor fibration} 
$\pi_s^Z:M_Z^s\rightarrow M_G$ if and only if
\begin{equation}
 \begin{cases}
  \beta =0 & (s=+), \\
 \alpha_-=\alpha_+,\ \operatorname{Im}\beta=0 & (s=-), \\
 \alpha_--\alpha_+=2\operatorname{Re}\beta,\ \operatorname{Im}\beta=0 
& (s=0).
 \end{cases}
\label{hor-tw}
\end{equation}
\par
Suppose $\varphi_s:\Sigma\rightarrow M^s_Z$ a horizontal 
immersion with respect to the twistor fibration 
$\pi_s^{SZ}:M_Z^s\rightarrow M_G$
and a local lift 
$\tilde{\varphi}_s=(u_-,u_+):U\rightarrow M_V$ of $\varphi_s$ on 
$U\subset\Sigma$ satisfying \eqref{hor-tw}. Then  \eqref{diff-u-+} is written as:
 \begin{gather*}
  du_-=
\begin{cases}
 i\alpha_-u_-+w_- & (s=+), \\ 
 i\alpha u_-+\beta u_++w_- & (s=-), \\
 i(\alpha+\beta)u_-+\beta u_++w_- & (s=0), 
\end{cases}
\\
 du_+=
 \begin{cases}
 i\alpha_+u_++w_+ & (s=+), \\
 \beta u_-+i\alpha u_++w_+ & (s=-), \\  
 \beta u_-+i(\alpha-\beta)u_++w_+ & (s=0),
 \end{cases}
 \end{gather*}
where 
$\alpha$ and $\beta$ are $\RR$-valued 
$1$-forms on $U$.

Let $(\tilde{u}_-,\tilde{u}_+)$ be an another local lift of $\varphi_s$ 
on $U$. By the definition of $M_Z^s$ 
(cf. \eqref{act1}, \eqref{act+}, \eqref{act-} and \eqref{act0}),
there are functions $\theta$ and $t$ on $U$ such that 
\begin{equation}
 (\tilde{u}_-,\tilde{u}_+)=
\begin{cases}
 e^{i\theta}(e^{it}u_-,e^{-it}u_+) & (s=+),\\
 e^{i\theta}(u_-,u_+)\begin{pmatrix} \cosh t & \sinh t \\
	  \sinh t & \cosh t
	 \end{pmatrix} & (s=-),\\
 e^{i\theta}(u_-,u_+)
\begin{pmatrix}
 1+it & t \\ t & 1-it
\end{pmatrix} & (s=0).
\end{cases}
\end{equation}
By taking differentiation and using \eqref{hor-tw}, we obtain:
\begin{gather*}
 d\tilde{u}_-=
 \begin{cases}
 e^{i(\theta+t)}(i(d\theta+dt+\alpha_-)u_-+w_-) 
& (s=-),\\
 e^{i\theta}(i(d\theta+\alpha)(\cosh t u_-+\sinh t u_+)
+(dt+\beta)(\sinh tu_-+\cosh tu_+) \\
\hspace{6.6cm} +\cosh tw_-+\sinh tw_+) & (s=+),\\
 e^{i\theta}(((-t+i)(d\theta+\alpha)+i(dt+\beta))u_-
+(it(d\theta+\alpha)+dt+\beta)u_+
\\
\hspace{7cm} +(1+it)w_-+tw_+
) & (s=0),
 \end{cases}
\\
 d\tilde{u}_+=
 \begin{cases}
 e^{i(\theta-t)}(i(d\theta-dt+\alpha_+)u_++w_+) 
& (s=-),\\
 e^{i\theta}(i(d\theta+\alpha)(\sinh t u_-+\cosh t u_+)
+(dt+\beta)(\cosh tu_-+\sinh tu_+) \\
\hspace{6.6cm} +\sinh tw_-+\cosh tw_+) & (s=+),\\
 e^{i\theta}((it(d\theta+\alpha)+dt+\beta)u_-
+((t+i)(d\theta+\alpha)-i(dt+\beta))u_+
\\
\hspace{7cm} +tw_-+(1-it)w_+
) & (s=0).
 \end{cases}
\end{gather*}
Given $\RR$-valued $1$-forms $\alpha_-$, $\alpha_+$, $\alpha$ and 
$\beta$ on $U$, on a small enough open subset $U'\subset U$, take solutions $(\theta,t)$
to the system of equations 
\begin{equation*}
 d\theta=
\begin{cases}
\displaystyle -\frac{1}{2}(\alpha_-+\alpha_+) & (s=+),\\
 -\alpha & (s=-,0),
\end{cases}
,\quad
 dt=
\begin{cases}
\displaystyle \frac{1}{2}(\alpha_+-\alpha_-) & (s=+),\\
 -\beta & (s=-,0).
\end{cases}
\end{equation*}
Then  $d\tilde{u}_-$ and $d\tilde{u}_+$ take value in 
$\{u_-, u_+\}^\perp$, respectively.
 Hence, by changing $(\tilde{u}_-,\tilde{u}_+)$ to $(u_-,u_+)$ and 
$U'$ to $U$, etc., we get: 
\begin{proposition}
 Let $\Sigma$ be a manifold and let $\varphi_s:\Sigma\rightarrow M_Z^s$ 
($s=+,-,0$) be an immersion which is horizontal with respect to the 
twistor fibration $\pi^{Z}_s:M_Z^s \rightarrow M_G$.
 Then  for each point $p\in\Sigma$, there exist an open subset 
$U\ni p$ and a local lift 
$\tilde{\varphi_s}=(u_-,u_+):U\rightarrow M_V$ of $\varphi_s$ satisfying
 \begin{equation}
  du_-=w_-,\quad du_+=w_+,
 \label{du-+}
 \end{equation}
where $w_-$ and $w_+$ are $\{u_-,u_+\}^\perp$-valued 
$1$-forms on $U\subset\Sigma$.
\end{proposition}

Given $(u_-,u_+)\in M_V$, we consider $3$ families of 
curves $\gamma^s_r$ ($s\in\{+,-,0\}$, $r\in\RR$) 
in $H_1^{2n+1}$:
\begin{align}
 \gamma^+_r(t) :&=
\begin{pmatrix} u_- & u_+ \end{pmatrix}
\begin{pmatrix} e^{it} & 0 \\ 0 & e^{-it} \end{pmatrix}
\begin{pmatrix} \cosh r \\ \sinh r \end{pmatrix}
=e^{it}\cosh ru_-+e^{-it}\sinh ru_+, \quad (t\in S^1) 
\label{gamma+}
\\
 \gamma^-_r(t) :&=
\begin{pmatrix} u_- & u_+ \end{pmatrix}
\begin{pmatrix} \cosh t & \sinh t \\ \sinh t & \cosh t \end{pmatrix}
\begin{pmatrix} \cosh r \\ i\sinh r \end{pmatrix}
\quad (t\in\RR)
\label{gamma-} \\
&=(\cosh r\cosh t+i\sinh r\sinh t)u_-+(\cosh r\sinh t+i\sinh r\cosh t)u_+,
\notag\\
 \gamma^0_r(t) :&=
\begin{pmatrix} u_- & u_+ \end{pmatrix}
\begin{pmatrix} 1+it & t \\ t & 1-it \end{pmatrix}
\begin{pmatrix} \cosh r \\ i\sinh r \end{pmatrix}
\quad (t\in\RR)
\label{gamma0} \\
&=(\cosh r+ite^r)u_-+(te^r+i\sinh r)u_+.
\notag
\end{align}
The image of the curve $\gamma^s_r$ lies in 
$H_1^{2n+1}\cap\operatorname{span}_\CC\{u_-,u_+\}=H_1^3$, so that 
the image of $\pi_{-}\circ\gamma^s_r$ is contained in a 
totally geodesic complex hyperbolic line $\pi_{-}(H_1^3)=\mathbb{CH}^1\subset \CHn$.
\par
The tangent vector field of the curve $\gamma_r^s$ ($s=+,-,0$) is 
\begin{equation}
  \frac{d}{dt}\gamma^s_r(t)=
  \begin{cases}
   i(e^{it}\cosh ru_--e^{-it}\sinh ru_+) & (s=+), \\
(\cosh r\sinh t+i\sinh r\cosh t)u_- \\
+(\cosh r\cosh t+i\sinh r\sinh t)u_+ 
& (s=-), \\
 e^r(iu_-+u_+) & (s=0).
  \end{cases}
\label{ddt-gam}
\end{equation}
Since
\begin{equation}
 \left\langle   \frac{d}{dt}\gamma^s_r(t), i\gamma^s_r(t)\right\rangle=
  \begin{cases}
   -\cosh 2r & (s=+), \\
   -\sinh 2r & (s=-), \\
   -e^{2r} & (s=0),
  \end{cases}
\label{ddtgm-igm}
\end{equation}
the horizontal part of \eqref{ddtgm-igm} with respect to 
the Hopf fibration $\pi_{-}:H^{2n+1}\rightarrow\CHn$ is:
\begin{align}
 \mathcal{H}  \frac{d}{dt}\gamma^s_r(t) &=  
\frac{d}{dt}\gamma^s_r(t)+
 \left\langle   \frac{d}{dt}\gamma^s_r(t), i\gamma^s_r(t)\right\rangle
 i\gamma^s_r(t)
\notag\\
&=\begin{cases}
-i\sinh 2r(e^{it}\sinh ru_-+e^{-it}\cosh ru_+) & (s=+),
\\
  \cosh 2r((\cosh r\sinh t-i\sinh r\cosh t)u_-
\\
+(\cosh r\cosh t-i\sinh r\sinh t)u_+) & (s=-),
\\
e^{2r}\left((te^r-i\sinh r)u_-+(\cosh r-ite^r)u_+\right) & (s=0).
 \end{cases}
\label{hor-ddtgm}
\end{align}
So a unit horizontal vector field $T_r^s$ along $\gamma^s_r$ is
\begin{align}
 T^s_r &=
\begin{cases}
\displaystyle \frac{1}{\sinh 2r} \mathcal{H}  \frac{d}{dt}\gamma^s_r(t) 
=  \frac{1}{\sinh 2r} \left(\frac{d}{dt}\gamma^s_r(t)
-i\cosh 2r \gamma^s_r(t)\right) & (s=+), \\
\displaystyle \frac{1}{\cosh 2r} \mathcal{H}  \frac{d}{dt}\gamma^s_r(t) 
=  \frac{1}{\cosh 2r} \left(\frac{d}{dt}\gamma^s_r(t)
-i\sinh 2r \gamma^s_r(t)\right) & (s=-), \\
\displaystyle \frac{1}{e^{2r}} \mathcal{H}  \frac{d}{dt}\gamma^s_r(t) 
=  \frac{1}{e^{2r}} \left(\frac{d}{dt}\gamma^s_r(t)
-ie^{2r} \gamma^s_r(t)\right) & (s=0), 
\end{cases}
\notag\\
&=\begin{cases}
-i(e^{it}\sinh ru_-+e^{-it}\cosh ru_+) & (s=+),
\\
  (\cosh r\sinh t-i\sinh r\cosh t)u_-
\\
+(\cosh r\cosh t-i\sinh r\sinh t)u_+ & (s=-),
\\
(te^r-i\sinh r)u_-+(\cosh r-ite^r)u_+ & (s=0).
 \end{cases}
\label{Trs}
\end{align}
With respect to the flat connection $D$ on $\CC_1^{n+1}$, 
we have
\begin{gather}
 D_{T^s_r}T^s_r=
\begin{cases}
 \displaystyle \frac{-2}{\sinh 2r}(e^{it}\sinh^3 ru_-
+e^{-it}\cosh^3 ru_+) & (s=+),
\\
\\
 \displaystyle \frac{1}{\cosh 2r}\Bigl((\cosh r\cosh t-i\sinh r\sinh t)u_-
& \\
\hspace{1.1cm} +(\cosh r\sinh t-i\sinh r\cosh t)u_+ & \\
\hspace{0.6cm} -\sinh 2r\bigl((\sinh r\cosh t+i\cosh r\sinh t)u_- & \\
\hspace{1.1cm} +(\sinh r\sinh t+i\cosh r\cosh t)u_+\bigr)\Bigr) & (s=-),
\\
 (e^{-r}-\sinh r-ite^r)u_--(te^r+i(e^{-r}+\cosh r))u_+ & (s=0).
\end{cases}
\label{DTT}
\end{gather}
Hence,  with respect to the induced Levi-Civita connection 
$\tilde{\nabla}$ on $H_1^{2n+1}$, we obtain:
\begin{gather}
 \tilde{\nabla}_{T^s_r}T^s_r
=D_{T^s_r}T^s_r+\langle D_{T^s_r}T^s_r,\gamma^s_r(t)\rangle
 \gamma^s_r(t) 
=\begin{cases}
 -2\coth 2r\, iT^+_r & (s=+), \\
 -2\tanh 2r\, iT^-_r & (s=+), \\  
 -2 iT^0_r & (s=0), \\  
 \end{cases}
\label{nTT}
\end{gather}
and the curvature $\kappa$ 
of $\pi_{-}\circ\gamma^s_r$ as a curve in $\CHn(-4)$ is 
$2\coth 2r$ ($s=+$) (resp. $2\tanh 2r$ ($s=-$) and $2$ ($s=0$)).
Also 
\begin{equation}
\cosh r'\gamma^s_r(t)+\sinh r'iT^s_r=\gamma^s_{r+r'}(t) 
\label{parallel}
\end{equation}
implies that $\{\pi_{-}\circ\gamma^s_r\}$ is a family of parallel 
curves of constant curvature $\kappa>2$ ($s=+$)
(resp. $\kappa<2$ ($s=-$) and $\kappa=2$ ($s=0$))
 in a totally geodesic complex hyperbolic line 
$\mathbb{CH}^1(-4)$ in $\CHn(-4)$.
Therefore, given $\pi_s^{SZ}\circ\pi^{VS}(u_-,u_+)\in M_Z^s$,
we have a family of parallel curves $\{\gamma_r^s(\RR)\}$ of 
constant curvature $2\coth 2r$ ($s=+$) (resp. $2\tanh 2r$ ($s=-$) and 
$2$ ($s=0$)) which lies in a totally geodesic complex hyperbolic line 
$\mathbb{CH}^1(-4)$ in $\CHn(-4)$.
\par
 Conversely, a curve $\gamma(\RR)$ of constant curvature $\kappa$, contained in a  
totally geodesic $\mathbb{CH}^1\subset \CHn$,  is obtained by this way, as $\gamma$ satisfies the ODE 
$\tilde{\nabla}_{\dot{\gamma}}\dot{\gamma}=\kappa\, i\,\dot{\gamma}$.  
\begin{theorem} \label{thm-cv-CHn}
 The twistor space $M^+_Z$ (resp. $M^-_Z$ and $M^0_Z$) is 
identified with the set of parallel family of curves $\{\pi_{-}\circ\gamma_r^+\}$ 
(resp. $\{\pi_{-}\circ\gamma_r^-\}$ and $\{\pi_{-}\circ\gamma_r^0\}$),  
each of which lies in a totally geodesic complex hyperbolic line $\mathbb{CH}^1\subset \CHn$,  
with constant curvature $\kappa>2$ (resp. $0\le \kappa<2$ and $\kappa=2$).
\end{theorem}

\section{Construction of Hopf real hypersurfaces in $\CHn$}
\label{Hopf-CHn}
In this section, we show how to construct Hopf hypersurfaces $M_+^{2n-1}$ 
(resp. $M_-^{2n-1}$ and $M_0^{2n-1}$) in $\CHn$  with Hopf curvature $\mu$ such that $|\mu|>2$ (resp. $|\mu|<2$ and $|\mu|=2$) from \textit{horizontal submanifolds} $\Sigma_+^{2n-2}$ (resp. $\Sigma_-^{2n-2}$ and $\Sigma_0^{2n-2}$) in $M_Z^+$ 
(resp. $M_Z^-$ and $M_Z^0$) with respect to the twistor fibration $\pi^Z_+\to M_G$ 
(resp. $\pi^Z_-$ and $\pi^Z_0$), as either real line or circle bundles.
\par
Given an immersion $\varphi_s:\Sigma_Z^s\rightarrow M_Z^s$ ($s=+,-$ or $0$),  we consider the pull-back bundle over $\Sigma_Z^s$ with respect to the 
fibration $\pi^{SZ}_s:M_S\rightarrow M_Z^s$ (cf. \eqref{act+}, 
\eqref{act-} and \eqref{act0}), namely 
\begin{equation}
 \Sigma_S:=\varphi_s^* M_S:=
\{(p,\pi^{VS}(u_-,u_+))\in\Sigma_Z^s\times M_S|\ 
\varphi(p)=\pi_s^{SZ}(\pi^{VS}(u_-,u_+))\}.
\label{pl-bdle}
\end{equation}
We have the following commutative diagram:
\begin{equation} \label{com-diag2}
 \xymatrix{
 \Sigma_S \ar[r]^{\psi_s} \ar[d]_{\pi_s^\Sigma} & M_S \ar[d]^{\pi^{SZ}_s} \\
 \Sigma_Z^s \ar[r]^{\varphi_s} & M^s_Z
}
\end{equation}
where $\pi_s^\Sigma(p,\pi^{VS}(u_-,u_+))=p$ and 
$\psi_s(p,\pi^{VS}(u_-,u_+))=\pi^{VS}(u_-,u_+)$. 
Next, take $\Sigma_V$ the pull-back bundle over 
$\Sigma_S$ with respect to the fibration 
$\pi^{VS}:M_V\rightarrow M_S$, that is, $\tilde{\pi}_s^{\Sigma}:\Sigma_V\to \Sigma_s$ is the  projection and $\tilde{\psi}_s:\Sigma_V\to M_V$ is the bundle map, with the commutative diagram 
\begin{equation} \label{com-diag3}
 \xymatrix{
 \Sigma_V \ar[r]^{\tilde{\psi}_s} 
\ar[d]_{\tilde{\pi}_s^\Sigma} & M_V \ar[d]^{\pi^{VS}} \\
 \Sigma_S \ar[r]^{\psi_s}  & M_S 
}.
\end{equation}
\par
According to \eqref{gamma+}, \eqref{gamma-} and \eqref{gamma0},  
we define new maps $\tilde{p}^s_r:M_V\rightarrow H_1^{2n+1}$ and 
$p^s_r:M_S\rightarrow\CHn$ by
\begin{gather}
 \tilde{p}^s_r(u_-,u_+)=
\begin{cases}
 \cosh r\,u_-+\sinh r\,u_+ & (s=+), \\
 \cosh r\,u_-+i\sinh r\,u_+ & (s=-,0), 
\end{cases}
\notag\\
p^s_r(\pi^{VS}(u_-,u_+))=
\pi_{-}(\tilde{p}^s_r(u_-,u_+))
\label{psr}
\end{gather}
and 
\begin{equation}
 \Phi^s_r:\Sigma_S\rightarrow\CHn,\quad
 \Phi^s_r=p^s_r\circ\psi_s,\quad
 \wtPhi^s_r:\Sigma_V\rightarrow\CHn,\quad
 \wtPhi^s_r=\tilde{p}^s_r\circ\tilde{\psi}_s.
\label{Phisr}
\end{equation}
Next, for each $p\in \Sigma_Z^s$, the fiber $(\pi^\Sigma_s)^{-1}(p)$ is mapped 
by $\Phi_r^s$ to the curve $\gamma_r^s(\RR)$ of constant curvature lying in a totally geodesic 
complex hyperbolic line $\pi_s^Z(\varphi(p))=\mathbb{CH}^1$ in $\CHn$, 
corresponding to $\varphi(p)$ as in Theorem  \ref{thm-cv-CHn}. 
We have the following commutative diagram:
\begin{equation}
\label{com-diag4}
 \xymatrix{
 \Sigma_V \ar[r]^{\tilde{\psi}_s} 
\ar[d]_{\tilde{\pi}_s^\Sigma} & M_V \ar[d]^{\pi^{VS}} \ar[r]^{\tilde{p}^s_r} & 
H_1^{2n+1} \ar[d]^{\pi_{-}} \\
 \Sigma_S \ar[r]^{\psi_s} \ar[d]_{\pi_s^\Sigma} 
& M_S \ar[d]^{\pi^{SZ}_s} \ar[r]^{p^s_r} & \CHn \\
 \Sigma_Z^s \ar[r]^{\varphi_s} & M^s_Z\ar[r]^{\pi^Z_s}  & M_G
}, \qquad s=+,-,0.
\end{equation}
\par
Consider an immersion $\varphi_s:\Sigma^{2n-2}=\Sigma_Z^s\rightarrow M_Z^s$ ($s=+,-$ or $0$) 
from a real $(2n-2)$-dimensional manifold, which is \textit{horizontal} with respect to the \textit{twistor fibration} 
$\pi_s^Z:M_Z^s\rightarrow M_G$. Using a local trivialization, we locally describe the map 
$\wtPhi^s_r:\Sigma_V\rightarrow H_1^{2n+1}$ as follows. Let 
$\widetilde{\Phi}^s_r:S^1\times\RR\times U\rightarrow H_1^{2n+1}$ be a map 
defined by (cf. \eqref{gamma+}, \eqref{gamma-} and \eqref{gamma0})
\begin{equation}
 \wtPhi^s_r(\theta,t,p)=
  \begin{cases}
   e^{i\theta}(e^{it}\cosh ru_-(p)+e^{-it}\sinh ru_+(p)) & (s=+), \\
e^{i\theta}((\cosh r\cosh t+i\sinh r\sinh t)u_-(p) \\
\ \ \ +(\cosh r\sinh t+i\sinh r\cosh t)u_+(p)) 
& (s=-), \\
e^{i\theta}((\cosh r+ite^r)u_-(p)+(te^r+i\sinh r)u_+(p)) & (s=0).
  \end{cases}
\label{wtPhi}
\end{equation}
Then  we have $\pi_{-}(\wtPhi_r^s(S^1\times\RR\times U))=
\Phi_r^s|_{(\pi_s^\Sigma)^{-1}(U)}((\pi_s^\Sigma)^{-1}(U))$. The derivatives of $\wtPhi_r^s$ with respect to $\theta$ and $t$ are (cf. \eqref{ddt-gam}):
\begin{align*}
 d\wtPhi_r^s(\partial/\partial \theta)&=i\wtPhi_r^s,
\\
 d\wtPhi_r^s(\partial/\partial t)&=
  \begin{cases}
   ie^{i\theta}(e^{it}\cosh ru_-(p)-e^{-it}\sinh ru_+(p)) & (s=+), \\
e^{i\theta}((\cosh r\sinh t+i\sinh r\cosh t)u_-(p) \\
\ \ \ +(\cosh r\cosh t+i\sinh r\sinh t)u_+(p)) 
& (s=-), \\
e^{i\theta}e^r(iu_-(p)+u_+(p)) & (s=0).
  \end{cases}
\end{align*}
According to \eqref{ddtgm-igm}, we get 
\begin{equation*}
 \langle d\wtPhi_r^s(\partial/\partial t), i\wtPhi_r^s\rangle=
  \begin{cases}
   -\cosh 2r & (s=+), \\
   -\sinh 2r & (s=-), \\
   -e^{2r} & (s=0).
  \end{cases}
\end{equation*}
By \eqref{hor-ddtgm}, the horizontal part of $d\wtPhi_r^s(\partial/\partial t)$ with respect to 
the Hopf fibration $\pi_{-}:H^{2n+1}_1\rightarrow\CHn$ is 
\begin{align}
 \mathcal{H} d\wtPhi_r^s(\partial/\partial t) &=  
d\wtPhi_r^s(\partial/\partial t) +
 \langle d\wtPhi_r^s(\partial/\partial t), i\wtPhi_r^s \rangle
 i\wtPhi_r^s
\notag\\
&=\begin{cases}
-ie^{i\theta}\sinh 2r(e^{it}\sinh ru_-+e^{-it}\cosh ru_+) & (s=+),
\\
  e^{i\theta}\cosh 2r((\cosh r\sinh t-i\sinh r\cosh t)u_-
\\
\hspace{1.5cm} +(\cosh r\cosh t-i\sinh r\sinh t)u_+) & (s=-),
\\
e^{2r}e^{i\theta}\left((te^r-i\sinh r)u_-+(\cosh r-ite^r)u_+\right) & (s=0).
 \end{cases}
\label{Hddt}
\end{align}
Also for $X\in T_p\Sigma$, using \eqref{du-+}, we have
\begin{equation}
 d\wtPhi^s_r(X)=
  \begin{cases}
   e^{i\theta}(e^{it}\cosh rw_-(X)+e^{-it}\sinh rw_+(X)) & (s=+), \\
e^{i\theta}((\cosh r\cosh t+i\sinh r\sinh t)w_-(X) \\
\ \ \ +(\cosh r\sinh t+i\sinh r\cosh t)w_+(X)) 
& (s=-), \\
e^{i\theta}((\cosh r+ite^r)w_-(X)+(te^r+i\sinh r)w_+(X)) & (s=0),
  \end{cases}
\label{dwtPhi-X} 
\end{equation}
where $w_-$ and $w_+$ are $\{u_-,u_+\}^\perp$-valued $1$-forms 
on $\Sigma$.
\begin{proposition}\label{horizontal}
 Let $\Sigma^{2n-2}$, $n\ge 2$, be a $(2n-2)$-dimensional manifold and 
let $\varphi_s:\Sigma\rightarrow M^s_Z$ $(s=+,-,0)$ be an immersion 
which is horizontal with respect to the twistor fibration 
$\pi_Z^s:M_Z^s\rightarrow M_G$. Then  the map $\Phi_r^s:\Sigma_S\rightarrow\CHn$ 
($r\in\RR-\{0\}$ when $s=+$, and  $r\in\RR$ when $s=-,0$)
defined by \eqref{Phisr} is an immersion if, and only if,  
for each $t\in\RR$, $p\in\Sigma$, $X\in T_p\Sigma$, \eqref{dwtPhi-X} 
does not vanish.
\end{proposition}
Suppose \eqref{dwtPhi-X} does not vanish for all $X\in T_p\Sigma$ 
at $p\in\Sigma$. Then
\begin{equation}
 N=d\pi_{-}(N')
\label{un}
\end{equation}
is a unit normal vector on a real hypersurface $\Phi^s_r(\Sigma_S)$ in $\CHn$, where 
\begin{equation}
 N'=
\begin{cases}
e^{i\theta}(e^{it}\sinh ru_-+e^{-it}\cosh ru_+) & (s=+),
\\
  ie^{i\theta}((\cosh r\sinh t-i\sinh r\cosh t)u_-
\\
\hspace{1.5cm} +(\cosh r\cosh t-i\sinh r\sinh t)u_+) & (s=-),
\\
ie^{i\theta}\left((te^r-i\sinh r)u_-+(\cosh r-ite^r)u_+\right) & (s=0).
 \end{cases}
\label{un'}
\end{equation}
Then
\begin{gather}
 \xi':=-iN'=
\begin{cases}
\displaystyle 
\frac{1}{\sinh 2r} \left(
d\wtPhi_r^s(\partial/\partial t) -\cosh 2r\,  i\wtPhi_r^s \right) & (s=+), \\
\\
\displaystyle 
\frac{1}{\cosh 2r} \left(
d\wtPhi_r^s(\partial/\partial t) -\sinh 2r \, i\wtPhi_r^s \right) & (s=-), \\
\\
\displaystyle 
 \frac{1}{e^{2r}} \left( d\wtPhi_r^s(\partial/\partial t) -e^{2r}  i\wtPhi_r^s \right) & (s=0), 
\end{cases}
\label{xi'}
\end{gather}
is a horizontal lift of the structure vector $\xi$ of 
$\Phi^s_r(\Sigma_S)$ in $\CHn$ (cf. \eqref{Trs}).
\par
Quite the same calculations as \eqref{DTT} and \eqref{nTT} yield 
\begin{equation}
 \tilde{\nabla}_{\xi'}\xi'=\mu N',\quad
\mu=\begin{cases}
 -2\coth 2r  & (s=+), \\
 -2\tanh 2r  & (s=+), \\  
 -2  & (s=0),
    \end{cases}
\label{nabla-xi'}
\end{equation}
and $\bar{\nabla}_\xi \xi=\mu N$, where $\bar{\nabla}$ is 
the Levi-Civita connection of $\CHn$.
Hence, $\nabla_\xi \xi=\phi A\xi=0$ and 
$\mu=\langle A\xi,\xi\rangle$ hold, so that $M$ is a 
\textit{Hopf hypersurface} in $\CHn(-4)$ with Hopf curvature $\mu$.
In addition, we see that  
$\cosh r'\wtPhi_r^s+\sinh r'N'=\wtPhi_{r+r'}^s$ (cf. \eqref{parallel}). That is, $r\mapsto \pi_{-}(\wtPhi_r^s(M^s_\Sigma))$ is a family of parallel Hopf real hypersurfaces, which may have singularities (focal points).
\begin{theorem} \label{th-construction} Let $\Sigma^{2n-2}$, $n\ge 2$, be a $(2n-2)$-dimensional manifold and 
let $\varphi_s:\Sigma\rightarrow M^s_Z$ $(s=+,-,0)$ be an immersion 
which is horizontal with respect to the twistor fibration 
$\pi_Z^s:M_Z^s\rightarrow M_G$. Assume that the map $\Phi_r^s:\Sigma_S\rightarrow\CHn$ 
$(r\in\RR-\{0\}$ when $s=+$, and  $r\in\RR$ when $s=-,0)$ defined by \eqref{Phisr} is an immersion.
\begin{enumerate}
\item The map $\Phi_r^s$ defines a Hopf real hypersurface in $\CHn$ such that $\vert\mu\vert>2$ for $s=+$, $\vert\mu\vert<2$ for $s=-$ and $\vert \mu\vert=2$ for $s=0$, respectively.
\item The family $r\mapsto \Phi_r^s(M_Z^s)$ is a family of parallel Hopf real hypersurfaces in $\CHn$, possibly with some focal points. 
\end{enumerate} 
\end{theorem}
\begin{remark} For $s=+$, according to Theorem \ref{th-Montiel}, for some $r\neq 0$, the set of focal points is a complex submanifold.
\end{remark}

\section{Classical Examples} \label{Examples} 
In this section, we use all our previous computations to obtain the classical examples of Hopf hypersurfaces in  $\CHn$, $n\ge 2$. 
\begin{example}\normalfont Given $n\ge 2$, take a natural number $0\le k\le n-1$, and recall $H_1^{2k+1}\subset \CC_1^{k+1}$ and 
$S^{2n-2k-1}\subset \CC^{n-k}$, so that we define the embedding 
\begin{gather*}
 \tilde{\psi}_+:\Sigma_V=H_1^{2k+1}\times S^{2n-2k-1}
\rightarrow M_V=V_{1,1}(\CC_1^{n+1}),\\
\left(  \tilde{u}_-, \tilde{u}_+\right) \mapsto
\left(
\begin{pmatrix}
 \tilde{u}_- \\ 0
\end{pmatrix},
\begin{pmatrix}
 0  \\ \tilde{u}_+
\end{pmatrix}
\right). 
\end{gather*}
Note that when $k=0$, $H_1^1\subset\mathbb{C}_1^1$ is nothing but the standard round $1$-sphere, but whose metric has the opposite sign to the usual one. Apart from that, everything works as in higher dimensions. 

We put $\Sigma_S:=\pi^{VS}(\Sigma_V)$ and $\Sigma_Z^+:=\pi^{SZ}_+(\Sigma_S)$, respectively.
By the definition of $M^+_Z$, $\Sigma_Z^+=\mathbb{CH}^k\times\mathbb{CP}^{n-k-1}$, 
and $\Sigma_S$ is the total space of the $S^1$-bundle over $\Sigma_Z^+$. 
We recall the maps $\Phi_r^+:\Sigma_S\rightarrow\CHn$ and 
$\wtPhi_r^+:\Sigma_V\rightarrow H_1^{2n+1}$ defined by \eqref{Phisr}, respectively.
Then  locally, $\wtPhi_r^+$ is described by \eqref{wtPhi} as:
\begin{equation*}
\wtPhi_r^+:T^2\times U\rightarrow H_1^{2n+1},\quad
 \wtPhi_r^+(\theta,t,p)=e^{i\theta}
 \begin{pmatrix}
  e^{it}\cosh r\,\tilde{u}_-(p) \\  e^{-it}\sinh r\,\tilde{u}_+(p) 
 \end{pmatrix},
\end{equation*}
where $U$ is an open subset of $\mathbb{CH}^k\times\mathbb{CP}^{n-k-1}$ 
and $p\mapsto (\tilde{u}_-(p),\tilde{u}_+(p))$ is a local cross section 
of the $T^2$-bundle $H_1^{2k+1}\times 
S^{2n-2k-1}\rightarrow\mathbb{CH}^k\times\mathbb{CP}^{n-k-1}$.
For a tangent vector 
$(X_-,X_+)\in T_p (\mathbb{CH}^k\times\mathbb{CP}^{n-k-1})$,
we compute 
\begin{equation*}
 d\wtPhi_r^+(0,0,(X_-,X_+))
=e^{i\theta}
\begin{pmatrix}
 e^{it}\sinh r X_- \\ e^{-it} \cosh r X_+
\end{pmatrix}.
\end{equation*}
Hence for $r\not=0$, 
$\wtPhi_r^+:T^2\times U\rightarrow H_1^{2n+1}$ and 
$\Phi_r^+:\Sigma_S\rightarrow\CHn$ are immersions.
By \eqref{Hddt}, \eqref{un'} and \eqref{xi'}, a unit normal vector 
of $\wtPhi_r^+$ is 
\begin{equation*}
 N'=i\xi'=
\frac{i}{\sinh 2r}
\mathcal{H} d\wtPhi_r^+(\partial/\partial t)
=e^{i\theta}
\begin{pmatrix}
 e^{it}\sinh r \tilde{u}_-(p) \\ e^{-it} \cosh r\tilde{u}_+(p)
\end{pmatrix}.
\end{equation*}
Then  \eqref{nabla-xi'} implies $A\xi=-2\coth 2r\xi$.
For horizontal tangent vectors $ d\wtPhi_r^+(0,0,(X_-,0))$ 
and $ d\wtPhi_r^+(0,0,(0,X_+))$, we obtain
\begin{gather*}
 D_{d\wtPhi_r^+(0,0,(X_-,0))}N'=
\begin{pmatrix}
 e^{it}\sinh r X_- \\ 0
\end{pmatrix}
=-d\wtPhi_r^+(0,0,(AX_-,0)),
\\
 D_{d\wtPhi_r^+(0,0,(0,X_+))}N'=
\begin{pmatrix}
 0 \\  e^{-it}\cosh r X_+
\end{pmatrix}
=-d\wtPhi_r^+(0,0,(0,AX_+)),
\end{gather*}
and
\begin{equation*}
 A(0,(X_-,0))=-\tanh r(0,(X_-,0)),
\quad
 A(0,(0,X_+))=-\coth r(0,(0,X_+)).
\end{equation*}
Note that $\wtPhi_0^+(\Sigma_V)$ is a totally geodesic 
submanifold $H_1^{2k+1}$ in $H_1^{2n+1}$ and 
$\Phi_0^+(\Sigma_S)$ is a totally geodesic complex submanifold 
$\mathbb{CH}^k$ in $\CHn$, so $\Phi_r^+(\Sigma_S)$ is 
nothing but the tube of radius $r$ over a totally geodesic 
$\mathbb{CH}^k$ in $\CHn$, $n\ge 2$, $0\le k\le n-1$ (cf. \cite{Mo}). In the case $k=0$, $\mathbb{CH}^0$ is a point, so that we obtain a geodesic sphere. 
$\Box$
\end{example}

\begin{example}\normalfont Recall the manifold 
\[ V_{1,1}(\RR_1^{n+1})=\{ (u_-,u_+)\in \RR_1^{n+1}\times \RR_1^{n+1} : 
\langle u_-,u_-\rangle =-1, \ \langle u_-,u_+\rangle=0, \ 
\langle u_+,u_+\rangle=1\}. 
\]
With it, let $\Sigma_V:=S^1\times V_{1,1}(\RR_1^{n+1})$, and as in \eqref{com-diag3}, 
we define a map $\tilde{\psi}_V:\Sigma_V\rightarrow M_V$ by $\tilde{\psi}_V(e^{i\theta},(u_-,u_+)):=e^{i\theta}(u_-,u_+)$.
The corresponding map  $\widetilde{\Phi}_r^-:\Sigma_V
\rightarrow H_1^{2n+1}$ ($r\in\RR$) defined by \eqref{wtPhi} is written as
\begin{gather*}
 \wtPhi^-_r(\theta,t,p)=
e^{i\theta}((\cosh r\cosh t+i\sinh r\sinh t)u_-(p) \\
\ \ \ +(\cosh r\sinh t+i\sinh r\cosh t)u_+(p)). 
\end{gather*}
When $r=0$, the image under 
$\pi_-\circ \wtPhi^-_0$ is a totally geodesic real 
projective space $\RR\HH^n$ and when $r>0$, 
 the image under 
$\pi_-\circ \wtPhi^-_r$ is a tube of radius $r$ over 
$\RR\HH^n$.
$\Box$ 
\end{example}

\begin{example} \normalfont For $n\geq 2$, let $\Sigma_V:=S^1\times\RR\times\CC^{n-1}$ and we define maps 
\begin{gather*}
u_-: \Sigma_V\rightarrow H_1^{2n+1},\quad 
u_-(\theta,t,\ppp)=e^{i\theta}
\begin{pmatrix}
 1+\|\ppp\|^2/2 \\ \|\ppp\|^2/2 \\ \ppp
\end{pmatrix},
\\
u_+:\Sigma_V\rightarrow S_2^{2n+1},\quad 
u_+(\theta,t,\ppp)=e^{i\theta}
\begin{pmatrix}
 -i\|\ppp\|^2/2 \\ i(1-\|\ppp\|^2/2) \\ -i\ppp
\end{pmatrix}.
\end{gather*}
If we have the map 
$\tilde{\psi}_0:\Sigma_V\rightarrow M_V$
of \eqref{com-diag3} 
as $\tilde{\psi}:=(u_-,u_+)$,
then the corresponding map 
$\widetilde{\Phi}_r^0:\Sigma_V
\rightarrow H_1^{2n+1}$ ($r\in\RR$) defined by \eqref{wtPhi} is
written by
\begin{gather*}
 \widetilde{\Phi}_r^0(\theta,t,\ppp)=
e^{i\theta}
\begin{pmatrix}
 e^r(it+\|\ppp\|^2/2)+\cosh r \\ e^r(it+\|\ppp\|^2/2)-\sinh r 
\\ e^r \ppp
\end{pmatrix}.
\end{gather*}
Clearly,  $\widetilde{\Phi}_r^0$ is an immersion and  
the image of $\pi_-\circ\widetilde{\Phi}_r^0$ satisfies 
the defining equation of a horosphere 
(\cite{Mo}, Example 6.2), for 
$z_0-z_1=e^{i\theta}e^r$ and $|z_0-z_1|^2=e^{2r}$.
$\Box$
\end{example}

\section{Hopf real hypersurfaces with $\vert\mu\vert=2$}\label{newexamples}

We develop here a method to construct  Hopf real hypersurfaces in $\CHn$ such that $\vert\mu\vert=2$. 

\begin{definition}\label{CKO} Let $G^{n-1}$ be a $(n-1)$-manifold, $n\ge 2$, and $\Omega$ a $\mathfrak{u}(1,n)$-valued form on $G$. We will say that $\Omega$ is a  \emph{CKO}-form if it can be written as
\[\Omega =\begin{pmatrix}
i \alpha_0 & \frac{i}{2}(\alpha_0-\alpha_1) & {}^t \xx-i\, ^ty_0 \\
\frac{i}{2}(\alpha_1-\alpha_0) & i\alpha_1 & -{}^t \xx+i\,^ty_1 \\
\xx+i y_0 & \xx+i y_1 & w_1+ iw_2
\end{pmatrix},
\] 
where $\alpha_0,\alpha_1\in \Omega^1_{G }\otimes \RR$,  $\xx_0,y_0,y_1\in \Omega^1_{G }\otimes \RR^{n-1}$, with $y_0$ and $y_1$ linearly dependent, $w_1\in \Omega^1_{G }\otimes \mathrm{Alt}_{n-1}(\RR)$, $w_2\in\Omega^1_{G }\otimes \mathrm{Sym}_{n-1}(\RR)$, and $\Omega^1_{G }$, $\mathrm{Alt}_{n-1}(\RR)$, $\mathrm{Sym}_{n-1}(\RR)$ denote the space of $1$-forms on $G $, the set of real alternate $(n-1)$-matrices and the set of real symmetric $(n-1)$-matrices, respectively. When $n=2$, we will take $w_1=0$. 
\end{definition}

\begin{proposition} Let $\Omega$ be a $\mathfrak{u}(1,n)$-valued \emph{CKO}-form on a $(n-1)$-dimensional manifold $G$, $n\ge 3$. Assume that $\Omega$ satisfies the following equations:
\begin{align*}
&\qquad  d\alpha_0+2{}^t\xx\wedge y_0 = 0, \quad 
d\alpha_1-2{}^t\xx\wedge y_1 =0, \quad {}^ty_0\wedge y_1=0, \\
& \left\{\begin{array}{l}
d\xx-y_0\wedge \alpha_0 -\frac12 y_1\wedge\alpha_1 +\frac12 y_1\wedge \alpha_0 +w_1\wedge \xx -w_2\wedge y_0=0, \\
dy_0+\xx\wedge \alpha_0+\xx\wedge \alpha_1 +w_2\wedge \xx +w_1\wedge y_0
=0, 
\end{array}
\right. \\
& \left\{\begin{array}{l}
d\xx-\frac12 y_0\wedge \alpha_0 + \frac12 y_0\wedge\alpha_1 -y_1\wedge \alpha_1 +w_1\wedge \xx -w_2\wedge y_1=0, \\
dy_1+\xx\wedge \alpha_0+\xx\wedge \alpha_1 +w_1\wedge y_1
+w_2\wedge \xx=0, 
\end{array}
\right. \\
& \left\{\begin{array}{l}
dw_1+ y_0\wedge{}^ty_0 - y_1\wedge{}^ty_1 +w_1\wedge w_1 
-w_2\wedge w_2=0, \\
dw_2+y_0\wedge {}^t\xx-\xx\wedge{}^ty_0 +\xx\wedge{}^ty_1-y_1\wedge{}^t\xx+w_1\wedge w_2+w_2\wedge w_1
=0. 
\end{array}
\right. 
\end{align*}
Then, for each $x_o\in G$ and each $B_o\in U(1,n)$, there exists a open neighbourhood $U$ of $x_o$ in $G$, and a smooth map $g:U\rightarrow U(1,n)$ such that $g(x_o)=B_o$ and $\Omega=g^{-1}dg_{x_o}$.
\end{proposition}
\begin{proof} These equations are equivalent to the \textit{Maurer-Cartan} equation for $\Omega$, see the textbook \cite{IL}.
 Note that ${}^ty_0\wedge y_1=0$ implies that for any $X,Y\in T_xG$, 
\begin{align*}
&0=({}^ty_0\wedge y_1)(X,Y)={}^ty_0(X)y_1(Y)-{}^ty_0(Y)y_1(X)\\
&=\langle y_0(X),y_1(Y)\rangle -\langle y_0(Y),y_1(X)\rangle.
\end{align*}
So, the $\RR^{n-1}$-valued $1$-forms $y_0$ and $y_1$ on $G$ are linearly dependent. This is trivial when $n=2$. 
\end{proof}
\begin{remark} \normalfont When $\dim G=1$, there is no need for integrability conditions on $\Omega$. $\Box$
\end{remark}

On $\RR_1^{n+1}$, $n\ge 2$, we take the standard metric 
$\langle x,y\rangle =-x_0y_0+\sum_{i=1}^nx_iy_i$, for $x,y\in \RR_1^{n+1}$. The hyperbolic space is the hyperquadric 
\[ \mathbb{RH}^n = \{ p\in \RR_1^{n+1} : \langle p,p\rangle=-1\}.\]
Take the light-like cone
\[ \mathcal{C}=\{\eta\in \RR_1^{n+1} : \langle \eta,\eta\rangle =0\}.\]
Given a point $\eta\in\mathcal{C}$, define 
\[ \mathbf{H}^{n-1}=\{p\in \mathbb{RH}^n : \langle p,\eta\rangle=1\}.\]
Let us see that this is a horosphere in $\mathbb{RH}^n$. We call $\chi:\mathbb{RH}^n\to \RR_1^{n+1}$ the position vector, so that the tangent space at $p$ is 
\[ T_p\mathbf{H}^{n-1} =\{ X\in \RR_1^{n+1} : \langle X,p\rangle = 0, \langle X,\eta \rangle=0\}.\]
It is easy to check that $N=-\chi-\eta$ is a globally defined unit normal to $\mathbf{H}^{n-1}$ on $\mathbb{RH}^n$. 
If $D$ and $\nabla$ are the Levi-Civita connection of $\RR_1^{n+1}$ and $\mathbb{RH}^n$, respectively, and $A_N$ is the shape operator associated with $N$, then for each $X \in T\mathbf{H}^{n-1}$, $A_NX=-\nabla_XN=-D_XN+\langle X,N\rangle\chi =X$, because $\eta$ is constant. As usual, we embed $\mathbb{RH}^n$ as a totally geodesic, totally real submanifod in $\CHn$. Then for each $\eta\in \mathcal{C}$, we have (up to isometries) a totally umbilical horosphere in $\mathbb{RH}^n$, which is totally real in $\mathbb{CH}^n$. Since the Lie group $SO(1,n)< U(1,n)$ acts by isometries on $\RR_1^{n+1}$, on $\mathbb{RH}^n$ and on the light-cone $\mathcal{C}$, it is enough to know one of these horospheres. 

Now we consider a  parametrization of the horosphere in $\mathbb{RH}^n$ for $\eta=\,^t(-1,1,0,\ldots,0)$, with $\mathbf{J}=\RR^+$:
\[\mathbf{f}:\mathbf{J}\times S^{n-2}\to \RR H^n, \quad 
\mathbf{f}(\lambda,\ppp)=\begin{pmatrix} 1+\frac{\lambda^2}{2} \\ -\lambda^2/2 \\ \lambda \ppp \end{pmatrix}.
\]
\begin{remark}\label{1-horosphere} \normalfont A $1$-dimensional horosphere is simply a horocycle. Then, we take the curve 
\[\mathbf{f}:\RR\to \mathbb{RH}^2,\quad \mathbf{f}(\lambda)=\begin{pmatrix}
1+\lambda^2/2 \\ -\lambda^2/2 \\ \lambda
\end{pmatrix}.
\]
We need to unify the notation in the following way: When $n\geq 3$, $\mathbf{J}=\RR^+$, $\ppp\in S^{n-2}$, $\dim T_{\ppp}S^{n-2}=n-2$, and we will usually denote $X\in T_{\ppp}S^{n-2}$. But when $n=2$, then $\mathbf{J}=\RR$, $S^0=\{\ppp=1\}\subset\RR$ and $T_{\ppp}S^{n-2}=\{0\}$, so that $X\in T_{\ppp}S^0$ can only be $X=0$. $\Box$
\end{remark}

We construct the following auxiliary maps:
\begin{gather*}(\tilde{v}_{-},\tilde{v}_+):\mathbf{J}\times S^{n-2}\rightarrow V_{1,1}(\CC_1^{n+1}), \\
\tilde{v}_-(\lambda,\ppp)=
\begin{pmatrix}
1+\frac{\lambda^2}{2} \\ -\lambda^2/2 \\ \lambda \ppp
\end{pmatrix}, \quad 
\tilde{v}_+(\lambda,\ppp)=
\begin{pmatrix}
\frac{-i}{2}\lambda^2 \\i\Big( \frac{\lambda^2}{2}-1\Big) \\ -i \lambda\ppp 
\end{pmatrix}. 
\end{gather*}

Recall that $U(1,n)$ acts naturally by isometries on $H_1^{2n+1}$ and on $S_2^{2n+1}$ by $(g,\mathbf{q})\mapsto g\mathbf{q}$. 
With this, given $G^{n-1}$ a $(n-1)$-dimensional manifold, let $g:G\to U(1,n)$ be a smooth immersion. As in Section \ref{Hopf-CHn}, take 
$\Sigma_V:=S^1\times G\times\RR\times \mathbf{J}
\times S^{n-2}$ and define the map
\begin{gather*}\tilde{\psi}_0:\Sigma_V \rightarrow M_V=V_{1,1}(\CC_1^{n+1}),\\ \tilde{\psi}_0(\theta,x,h,\lambda,\ppp):=e^{i\theta}g(x)\Big( (1-ih)\tilde{v}_{-}(\lambda,\ppp)-h\tilde{v}_+(\lambda,\ppp),
-h\tilde{v}_{-}(\lambda,\ppp)+(1+ih)\tilde{v}_+(\lambda,\ppp)\Big).
\end{gather*}
We obtain the maps $\psi_0$, $\varphi_0$ and the projections, as in Section \ref{Hopf-CHn}, having a similar diagram to \eqref{com-diag}. In order to compute what we really need, we recall the projection \eqref{psr} for $s=r=0$, so that 
\begin{gather}
\tilde{\Psi}:\Sigma_V\rightarrow H_1^{2n+1}, \quad 
\tilde{\Psi}:=\tilde{p}_0^0\circ\tilde{\psi}_0, 
\nonumber \\
\tilde{\Psi}(\theta,x,h,\lambda,\ppp)= 
e^{i\theta}g(x) \begin{pmatrix} 1+\lambda^2/2-i h \\ -\lambda^2/2+i h \\  \lambda \ppp \end{pmatrix}.  \label{position-vector} 
\end{gather}
By using the Hopf map $\pi_{-}:H_1^{2n+1}\to \CHn$, we would like to obtain a real hypersurface  
\begin{equation}\label{theexample}
\xymatrix{
 \Sigma_V \ar[r]^{\tilde{\Psi}} \ar[d]_{\pi_{-}} & H_1^{2n+1} \ar[d]^{\pi_{-}} \\
 M \ar[r]^{\Psi_g} & \CHn
}
\end{equation}	
where $M=\pi_{-}(\Sigma_V)\cong G\times\RR\times\mathbf{J}\times S^{n-2}$. The map $\Psi_g$ is the desired example.

\begin{remark} \normalfont Intuitively, the images by $\tilde{\Psi}$ of the slices $x=$constant and $(s,\lambda,\ppp)=$constant have to be transversal each other in order to obtain a immersion that will produce a real hypersurface in $\CHn$. 
$\Box$
\end{remark}

We compute the horizontal part of the differential of $\tilde{\Psi}$ with respect to $\pi_{-}$, by using suitable identifications. A unit vector of the vertical direction is 
\[ d\tilde{\Psi}(\partial_{\theta})=i\tilde{\Psi} =e^{i\theta} g(x)\begin{pmatrix}
h+i	\big(1+\lambda^2/2\big) \\ -h-\frac{i}{2}\lambda^2 \\
 i \lambda \ppp
\end{pmatrix}.\]
Simple computations show
\[
d\tilde{\Psi}(\partial_h) = e^{i\theta} g(x) \begin{pmatrix} -i \\ i \\ 0
\end{pmatrix}, \quad 
\langle d\tilde{\Psi}(\partial_h),i\tilde{\Psi}\rangle  = 1.
\]
For the sake of simplicity, we denote 
\[\ww={}^t(-1,1,0,\ldots,0),\quad d\tilde{\Psi}(\partial_h) = e^{i\theta} g(x)\,i \ww.\]
From here,
\begin{align*}
& \mathcal{H} d\tilde{\Psi}(\partial_h) =  d\tilde{\Psi}(\partial_h) +
\langle d\tilde{\Psi}(\partial_h),i\tilde{\Psi}\rangle i \tilde{\Psi} 
= d\tilde{\Psi}(\partial_h) + i \tilde{\Psi}
= e^{i\theta} g(x) \begin{pmatrix} h+\frac{i}{2}\lambda^2 \\ 
-h + i\big( 1- \lambda^2/2\big) \\ i \lambda \ppp \end{pmatrix}.
\end{align*}
Given $X\in T_{\ppp}S^{n-2}$,  
\begin{gather*} d\tilde{\Psi}(\partial_{\lambda}) = e^{i\theta} g(x)\begin{pmatrix} \lambda \\ -\lambda \\ \ppp \end{pmatrix}, \quad
d\tilde{\Psi}(X)=e^{i\theta} g(x) \begin{pmatrix} 0 \\ 0 \\  \lambda X\end{pmatrix}, \\ 
\langle d\tilde{\Psi}(\partial_{\lambda}), i \tilde{\Psi}\rangle=0, \quad \mathcal{H}d\tilde{\Psi}(\partial_{\lambda})=d\tilde{\Psi}(\partial_{\lambda}),  \quad
\langle d\tilde{\Psi}(X),i \tilde{\Psi}\rangle = 0, \quad \mathcal{H}d\tilde{\Psi}(X)=d\tilde{\Psi}(X).
\end{gather*}
The induced metric on $\RR\times\RR^+\times S^{n-2}$ is 
\begin{gather*}
\langle \mathcal{H}d\tilde{\Psi}(\partial_h),\mathcal{H}d\tilde{\Psi}(\partial_h)\rangle = 1, \quad
\langle
d\tilde{\Psi}(\partial_{\lambda}), 
d\tilde{\Psi}(\partial_{\lambda})\rangle = 1, \quad
\langle
d\tilde{\Psi}(X),
d\tilde{\Psi}(X)\rangle = \lambda^2\|X\|^2, \\
\langle \mathcal{H}d\tilde{\Psi}(\partial_h),d\tilde{\Psi}(X)\rangle = 
\langle \mathcal{H}d\tilde{\Psi}(\partial_h),d\tilde{\Psi}(\partial_{\lambda})\rangle =
\langle 
d\tilde{\Psi}(\partial_{\lambda}),d\tilde{\Psi}(X)\rangle =  0.
\end{gather*}
In addition, $\mathcal{H} d\tilde{\Psi}(\partial_h)$ satisfies
\begin{gather*}
D_{ \mathcal{H} d\tilde{\Psi}(\partial_h)} \mathcal{H}d\tilde{\Psi}(\partial_h) 
= D_{d\tilde{\Psi}(\partial_h) + i \tPsi} \mathcal{H}d\tilde{\Psi}(\partial_h) 
= D_{d\tilde{\Psi}(\partial_h) + i \tPsi} 
\left[
e^{i\theta}  g(x) \begin{pmatrix} h+\frac{i}{2}\lambda^2 \\ 
-h+ i\big( 1- \lambda^2/2\big) \\ i \lambda \ppp \end{pmatrix}
\right] \\ 
=  e^{i\theta}g(x) \begin{pmatrix} 1 \\ -1 \\ 0 \end{pmatrix}
+ i e^{i\theta} g(x) \begin{pmatrix} h+\frac{i}{2}\lambda^2 \\ 
-h + i\big( 1- \lambda^2/2\big) \\ i \lambda \ppp \end{pmatrix} 
=  e^{i\theta} g(x) \begin{pmatrix} 1 -\lambda^2/2 +ih \\ 
-2 +\lambda^2/2 -i h\\ -\lambda \ppp \end{pmatrix}.
\end{gather*}
If $\tilde{\nabla}$ is the Levi-Civita connection on $H_1^{2n+1}$ induced from the flat connection of $\mathbb{C}_1^{n+1}$, we compute
\begin{gather*}
\tilde{\nabla}_{\mathcal{H} d\tilde{\Psi}(\partial_h)} \mathcal{H} d\tilde{\Psi}(\partial_h)  = 
D_{\mathcal{H} d\tilde{\Psi}(\partial_h)} \mathcal{H} d\tilde{\Psi}(\partial_h) 
+\langle D_{\mathcal{H} d\tilde{\Psi}(\partial_h)} \mathcal{H} d\tilde{\Psi}(\partial_h),\tilde{\Psi}\rangle\tilde{\Psi} \\
=e^{i\theta} g(x) \begin{pmatrix}1-\frac{\lambda^2}{2}+ih \\ -2+\frac{\lambda^2}{2} -ih \\ -\lambda\ppp \end{pmatrix} 
-e^{i\theta} g(x) \begin{pmatrix} 1 +\frac{\lambda^2}{2}-ih \\ -\frac{\lambda^2}{2}+ih \\ \lambda \ppp \end{pmatrix} 
=e^{i\theta} \begin{pmatrix} -\lambda^2+2hi \\ -2+\lambda^2-2h i \\ -2\lambda\ppp \end{pmatrix}
=2 i \mathcal{H} d\tilde{\Psi}(\partial_h). 
\end{gather*}
This shows that the image of each integral curve of $\mathcal{H}d\tilde{\Psi}(\partial_h)$ under the Hopf projection $\pi_{-}$ is contained in a totally geodesic $\mathbb{CH}^1$ in $\mathbb{CH}^n$ as a circle of constant curvature $2$. Recall that this is the case of the integral curves of the structure vector field $\xi$ of Hopf hypersurfaces $M^{2n-1}$ in $\mathbb{CH}^n(-4)$ with $A\xi=2\xi$. 
\par 
Next, we consider the differential of the map $g:G\to U(1,n)$. The pull back of the Maurer-Cartan form on $U(1,n)$ by $g$ is a $\mathfrak{u}(1,n)$-valued 1-form on $G$, which can be written as follows
\[\Omega:=g^{-1}dg =\left(
\begin{array}{c|c|c}
i\alpha_0 & \overline{\beta} &  {}^tx_0-i\,{}^ty_0   \\ \hline 
\beta & i\alpha_1 & -\,{}^tx_1+i\,{}^ty_1    \\ \hline 
x_0+i\,y_0 & x_1+i\,y_1 & w_1+i\, w_2
\end{array}\right), 
\]
where $\alpha_0,\alpha_1\in \Omega^1_{G }\otimes \RR$, $\beta\in \Omega^1_{G }\otimes \mathbb{C}$, $x_0,x_1,y_0,y_1\in \Omega^1_{G }\otimes \RR^{n-1}$, $w_1\in \Omega^1_{G }\otimes \mathrm{Alt}_{n-1}(\RR)$, $w_2\in\Omega^1_{G }\otimes \mathrm{Sym}_{n-1}(\RR)$, and $\Omega^1_{G }$, $\mathrm{Alt}_{n-1}(\RR)$, $\mathrm{Sym}_{n-1}(\RR)$ denote the space of $1$-forms on $G $, the set of real alternate $(n-1)$-matrices and the set of real symmetric $(n-1)$-matrices, respectively. When $n=2$, then $w_1=0$. We also denote $w=w_1+iw_2$, and $z_l=x_l+i y_l$ for $l=0,1$, and its conjugate as 	$\overline{z_l}=x_l-iy_l$, $l=0,1$. Given $Y\in T_xG ^{n-1}$, we have
\begin{align*}
d\tilde{\Psi}_x(Y) = & e^{i\theta}g(x) 
\left(
\begin{array}{c|c|c}
i\alpha_0(Y) & \overline{\beta}(Y) & {}^t\overline{z_0}(Y) \\ \hline 
\beta(Y) & i\alpha_1(Y) & -{}^t\overline{z_1}(Y) \\ \hline 
z_0(Y) & z_1(Y) & w(Y) 
\end{array}\right)
\begin{pmatrix} 1+\lambda^2/2 -i h \\ -\lambda^2/2 +ih \\ \lambda \ppp \end{pmatrix}
=e^{i\theta}g(x) \begin{pmatrix}
 \tilde{\Psi}_1 \\ \tilde{\Psi}_2 \\ \tilde{\Psi}_3
\end{pmatrix},
\end{align*}
where
\begin{align*}
\tilde{\Psi}_1 = &\Big( h+i\Big(1+\frac{\lambda^2}{2}\Big)  \Big) \alpha_0(Y) 
+\Big(-\frac{\lambda^2}{2} +ih\Big)\overline{\beta}(Y) 
+\lambda \langle x_0(Y), \ppp \rangle + i\lambda\langle y_0(Y),\ppp\rangle,
\\
\tilde{\Psi}_2 = & \Big(1+\frac{\lambda^2}{2}-ih\Big)\beta(Y) -\Big(h+i\frac{\lambda^2}{2}\Big)\alpha_1(Y) 
-\lambda \langle x_1(Y),\ppp\rangle +i \lambda \langle y_1(Y),\ppp\rangle, \\
\tilde{\Psi}_3 = & \Big(1+\frac{\lambda^2}{2}-ih\Big) z_0(Y) +\Big(-\frac{\lambda^2}{2}+i h\Big) z_1(Y) 
+\lambda w(Y)\ppp. 
\end{align*}

Now, we assume that $\tPsi$ and $\Psi$ are immersions, cf. \eqref{theexample}. According to Section \ref{Hopf-CHn}, we obtain the next equation  
\[\mathcal{H}d\tPsi(\partial_h)=\xi'  = -(JN)'=-iN',\]
where $\xi'$ is the horizontal lift of the structure vector $\xi$ of $\psi:M^{2n-1}\to \CHn(-4)$. The horizontal lift $N'$ of a unit normal field $N$ of $M^{2n-1}$ is
\begin{equation}\label{Nprime} 
N' = i\xi'=i \mathcal{H}d\tPsi(\partial_h) = e^{i\theta}g(x) 
\begin{pmatrix} -\frac{\lambda^2}{2}+i h \\  \frac{\lambda^2}{2}-1 - i h \\ -\lambda \ppp 
\end{pmatrix}.
\end{equation}
Given $X\in T_{\ppp}S^{n-2}$ and $Y\in T_xG^{n-2}$, a long but straightforward computation shows 
\begin{gather*} \langle \mathcal{H}d\tPsi(\partial_h),N'\rangle = 
\langle d\tPsi(\partial_{\lambda}),N'\rangle = \langle d\tPsi(X),N'\rangle = 0, \\
\langle d\tPsi(Y), N'\rangle = \lambda \langle x_1(Y)-x_0(Y),\ppp\rangle
+h(\alpha_1(Y)-\alpha_0(Y)-2\mathrm{Im}\beta(Y))+\mathrm{Re}\beta(Y).
\end{gather*}
From this, by fixing some variables and moving others, we obtain that $ \langle d\tPsi(Y),N'\rangle = 0$  holds on $\Sigma_V$ if, and only if, 
\begin{equation} \label{gdg}
\mathrm{Re}\beta = 0, \quad \mathrm{Im}\beta = \frac12( \alpha_1-\alpha_0), \quad x_0=x_1.
\end{equation}
Putting $\xx:=x_0=x_1$, $\beta = \frac{i}{2}(\alpha_1-\alpha_0)$, the form $\Omega=g^{-1}dg$ simplifies to 
\[\Omega=g^{-1}dg = \begin{pmatrix}
i \alpha_0 & \frac{i}{2}(\alpha_0-\alpha_1) & {}^t \xx-i\, ^ty_0 \\
\frac{i}{2}(\alpha_1-\alpha_0) & i\alpha_1 & -{}^t \xx+i\,^ty_1 \\
\xx+i y_0 & \xx+i y_1 & w_1+ iw_2
\end{pmatrix}.
\]
In other words, $\Omega=g^{-1}dg$ is a CKO-form. 

The next step is to compute the shape operator of $\Psi=\pi_{-}\circ\tilde{\Psi}$, i.~e., the shape operator of the induced real hypersurface in $\CHn$, so by abuse of notation, we will also denote it by $A$. 
It suffices to study the shape operator of $\tilde{\Psi}$. By \eqref{position-vector} and \eqref{Nprime}, we describe 
\[N' = -\tilde{\Psi} -e^{i\theta} g(x)\ww.\] 
For horizontal vectors $d\tilde{\Psi}(\partial_{\lambda})$ and $d\tilde{\Psi}(X)$, $X\in T_{\ppp}S^{n-2}$, simple computations give 
\begin{align*}
& D_{d\tilde{\Psi}(\partial_{\lambda})} N'=D_{d\tilde{\Psi}(\partial_{\lambda})} (-\tilde{\Psi}) = -d\tilde{\Psi}(\partial_{\lambda}), \quad 
D_{d\tilde{\Psi}(X)} N'= D_{d\tilde{\Psi}(X)} (\tilde{\Psi}) =  -d\tilde{\Psi}(X). 
\end{align*}
The Gauss' equation $D_{d\tilde{\Psi}(Z)} N' = - d\tilde{\Psi}(AZ)$ implies
\[ A\partial_{\lambda} = \partial_\lambda, \quad AX=X, \ X\in T_{\ppp}S^{n-2}.\]
This shows $\dim V_1\ge n-1$. For the horizontal vector $\xi'=\mathcal{H}d\tilde{\Psi}(\partial_h)=d\tilde{\Psi}(\partial_h)+i\tilde{\Psi}$, we obtain 
\begin{gather*} \mathcal{H}D_{\mathcal{H}d\tilde{\Psi}(\partial_h)} N' = \mathcal{H}D_{d\tilde{\Psi}(\partial_h)+i\tilde{\Psi}}N' =
\mathcal{H}D_{d\tilde{\Psi}(\partial_h)}N'+\mathcal{H}D_{i\tilde{\Psi}}N'\\
=\mathcal{H}D_{d\tilde{\Psi}(\partial_h)}(-\tPsi)+iN'
= -\mathcal{H}d\tilde{\Psi}(\partial_h)+iN'=-2\xi'.
\end{gather*}
Since $\tilde{\Psi}$ is an immersion, there is a unique $\hat{\xi}\in T\Sigma_V$ such that $d\tilde{\Psi}(\hat{\xi})=\xi'$. Together with $\mathcal{H} D_{\mathcal{H}d\tilde{\Psi}(\partial_h)}N'=-d\tilde{\Psi}(A \hat{\xi})$, we have
\[ A\hat{\xi} = 2\hat{\xi}.\]
Given $Y\in T_xG^{n-1}$, we compute
\begin{align*}
 D_{\mathcal{H}d\tilde{\Psi}(Y)} N' &= D_{d\tilde{\Psi}(Y)+\langle d\tilde{\Psi}(Y),i\tilde{\Psi}\rangle i \tilde{\Psi}} N' = -d\tilde{\Psi}(Y)-e^{i\theta} dg(Y)\ww
+ \langle d\tilde{\Psi}(Y), i\tilde{\Psi}\rangle i N' \\
& = -d\tilde{\Psi}(Y) -e^{i\theta} g(x) \Omega(Y)\ww 
-\langle d\tilde{\Psi}(Y), i\tilde{\Psi}\rangle \xi'.
\end{align*}
By taking the horizontal part, 
\[\mathcal{H} D_{\mathcal{H}d\tilde{\Psi}(Y)} N'  = - Hd\tilde{\Psi}(Y) -\mathcal{H}\big( e^{i\theta} g(x)\Omega(Y)\ww\big) 
- \langle d\tilde{\Psi}(Y), i\tilde{\Psi}\rangle \xi'.\]
We finally obtain 
\[Y\in T_xG,\quad  AY = Y + \big(d\tilde{\Psi}\big)^{-1} \mathcal{H}\Big(e^{i\theta} g(x)\Omega(Y)\ww 
\Big) +\langle d\tilde{\Psi}(Y),i\tilde{\Psi}\rangle \hat{\xi}. 
\]

If $A$ is the Weingarten endomorphism of a real hypersurface in $\CHn$, then we denote the eigenspace associated with a principal curvature $\rho$ by $V_{\rho}$. We have proved the following result. 

\begin{theorem}\label{Axi2xi} Let $G^{n-1}$ be a $(n-1)$-manifold, $n\ge 2$, and $g:G^{n-1}\rightarrow U(1,n)$ an immersion with associated $\mathfrak{u}(1,n)$-valued form $\Omega=g^{-1}dg$. Construct the maps $\tilde{\Psi}$ and $\Psi_g$ as in \eqref{position-vector} and \eqref{theexample}, and assume that they are immersions. The following are equivalent:
\begin{enumerate}
\item $\Omega$ is  a \emph{CKO}-form. 
\item The map $\Psi_g:M^{2n-1}\rightarrow \CHn$ is a real Hopf hypersurface such that its Weingarten endomorphism $A$ satisfies $A\xi=2\xi$ and $n-1\le \dim V_{1}$. 
\end{enumerate}
\end{theorem}

Our next target is to study those real hypersurfaces in $\mathbb{CH}^2$ which arise from 1-parameter subgroups of $U(1,2)$. Let $g:\tilde{I}\subset \RR\to U(1,2)$ be a regular curve such that $\Omega=g^ {-1}g'$ is a CKO-form, namely 
\begin{equation}\label{1-parameter}
\Omega = g^{-1}g' = \begin{pmatrix}
i \alpha_0 & \frac{i}{2}(\alpha_0-\alpha_1) & \xx-i\, y_0 \\
\frac{i}{2}(\alpha_1-\alpha_0) & i\alpha_1 & -\xx+i\,y_1 \\
\xx+i y_0 & \xx+i y_1 & i w
\end{pmatrix},
\end{equation}
where $\alpha_0$, $\alpha_1$, $\xx$, $y_0$, $y_1$ and $w$ are real functions defined on the interval $\tilde{I}$. 
\begin{remark} \label{1-param}\normalfont The curve $g$ is a 1-parameter subgroup of $U(1,2)$ if, and only if, all functions $\alpha_0$, $\alpha_1$, $\xx$, $y_0$, $y_1$ and $w_2$ are constant, and not simultaneously zero (cf. \cite{H}). $\Box$
\end{remark}

As in previous pages, consider $\bar{M}=S^1\times \RR\times\RR^+\times\tilde{I}$ and the map
\begin{gather*}\tilde {\Psi}:\bar{M}\to H_1^5, \quad 
\tilde{\Psi}(\theta,x,h,\lambda)=
\tilde{\Psi}(\theta,s,r,t)=e^{i\theta}g(x)\begin{pmatrix} 1+\lambda^2/2 -i h \\ -\lambda^2/2+ i h \\ \lambda \end{pmatrix}.
\end{gather*}
For $M^3=\pi_{-}(\bar{M})$, assume that the associated map $\Psi:M^3\to \mathbb{CH}^2$ provides a real hypersurface. We will use the proof of Theorem \ref{Axi2xi}, we set a (horizontal lift of a) unit normal $N$ of $M^3$ in $\mathbb{CH}^2$,
\begin{equation} N'=e^{i\theta}g(x)\begin{pmatrix}
-\lambda^2/2+ih \\ \lambda^2/2 -1 -i h \\ -\lambda \end{pmatrix},
\end{equation}
and we also obtain
\[ A\hat{\xi}=2\hat{\xi}, \quad A\partial_{\lambda}=\partial_{\lambda}.\] 
We are going to compute the other principal curvature in a simplified, but rather general situation. To do so, we assume that $g$ is a 1-parameter subgroup of $U(1,2)$, as in Remark \ref{1-param}. Denote again $\langle,\rangle$ the induced metric on $\bar{M}$. Since $\{\partial_{\lambda},\hat{\xi}\}$ are orthonormal, we define a tangent vector $W$ by
\[d\tilde{\Psi}(W)=\mathcal{H}d\tilde{\Psi}(\partial_x)-\langle d\tilde{\Psi}(\partial_x),d\tilde{\Psi}(\partial_{\lambda})\rangle d\tilde{\Psi}(\partial_{\lambda}) - \langle d\tilde{\Psi}(\partial_x),d\tilde{\Psi}(\xi)\rangle d\tilde{\Psi}(\xi).
\]
Some long computations give
\[ d\tilde{\Psi}(W)=\frac12\left(\lambda(2w-\alpha_0-\alpha_1)+2y_0+3\lambda^2(y_0-y_1)\right)i\,d\tilde{\Psi}(\partial_{\lambda}).
\]
\begin{remark} \normalfont If $y_0=y_1=0$ and $\alpha_0+\alpha_1=2w\neq 0$, then the map $\psi:M^3\to\mathbb{CH}^2$ is not an immersion.$\Box$
\end{remark}

Now,
\begin{align*}  D_{d\tilde{\Psi}(W)}N'&
= D_{\mathcal{H}d\tilde{\Psi}(\partial_x)}N'
-\langle d\tilde{\Psi}(\partial_x),d\tilde{\Psi}(\partial_{\lambda})\rangle D_{d\tilde{\Psi}(\partial_{\lambda})}N'
-\langle d\tilde{\Psi}(\partial_x),d\tilde{\Psi}(\xi)\rangle D_{d\tilde{\Psi}(\xi)}N' \\
&= D_{d\tilde{\Psi}(\partial_x)}N'
-\langle d\tilde{\Psi}(\partial_x),d\tilde{\Psi}(\partial_{\lambda})\rangle D_{d\tilde{\Psi}(\partial_{\lambda})}N'
-\langle d\tilde{\Psi}(\partial_x),i\tilde{\Psi}\rangle 
D_{d\tilde{\Psi}(\xi)}N' \\ 
& = D_{d\tilde{\Psi}(\partial_x)}N'
+\langle d\tilde{\Psi}(\partial_x),d\tilde{\Psi}(\partial_{\lambda})\rangle d\tilde{\Psi}(\partial_{\lambda})
+2\langle d\tilde{\Psi}(\partial_x),i\tilde{\Psi}\rangle d\tilde{\Psi}(\xi).
\end{align*}
With this, 
\begin{align*}
& - d\tilde{\Psi}(AW) = \mathcal{H}D_{d\tilde{\Psi}(W)}N'\\
& = \mathcal{H} D_{\mathcal{H}d\tilde{\Psi}(\partial_x)}N'
+\langle d\tilde{\Psi}(\partial_x),d\tilde{\Psi}(\partial_{\lambda})\rangle d\tilde{\Psi}(\partial_{\lambda})
+2\langle d\tilde{\Psi}(\partial_x),i\tilde{\Psi}\rangle  d\tilde{\Psi}(\xi) \\
&=\frac{-1}{2}\left(\lambda(2w-\alpha_0-\alpha_1)+2y_1+3\lambda^2(y_0-y_1)\right)i\,d\tilde{\Psi}(\partial_{\lambda}).
\end{align*}
All together,
\begin{equation} AW=\frac{\lambda(2w-\alpha_0-\alpha_1)+2y_1+3\lambda^2(y_0-y_1)}{\lambda(2w-\alpha_0-\alpha_1)+2y_0+3\lambda^2(y_0-y_1)} W.
\end{equation}
This principal curvature can be written 
\begin{equation}\label{principalcurvature}\rho = \frac{a}{b},\quad  a=\lambda(2w-\alpha_0-\alpha_1)+2y_1+3\lambda^2(y_0-y_1), \ 
b=a+2(y_0-y_1).
\end{equation}
Clearly, $\rho$ depends only on $\lambda$. Assume that $\rho$ is a constant function. By taking derivation w.r.t. $\lambda$, then $0=\rho'=\frac{a'b-ab'}{b^2}$, so that $0=a'b-ab'=2a(y_0-y_1)$. If $y_0\neq y_1$, then $0=\lambda(2w-\alpha_0-\alpha_1)+2y_1+3\lambda^2(y_0-y_1)$ for infinitely many $\lambda\in\RR$, which is a contradiction. Then, $y_0=y_1$, and then $\rho=1$ everywhere. Finally, $M^3$ is an open subset of a horosphere, because this is the only real hypersurface in $\mathbb{CH}^2$ such that $A\xi=2\xi$ and $AX=X$ for any $X\perp \xi$. 

\begin{theorem} \label{horo} Let $g:\tilde{I}\subset\RR\rightarrow U(1,2)$ be a 1-parameter subgroup of $U(1,2)$ such that $\Omega=g^{-1}g'$ is a \emph{CKO}-form and its associated map $\Psi:M^3\to \mathbb{CH}^2$ is a Hopf real hypersurface. The following are equivalent:
\begin{enumerate}
\item $\psi:M^3\to \mathbb{CH}^2$ is an open subset of a horosphere.
\item The principal curvature $\rho$ in \eqref{principalcurvature} is a constant function.
\item $y_0=y_1$ in \eqref{1-parameter}.
\end{enumerate}
\end{theorem}

\par 
\begin{trivlist} 

\item[]
Jong Taek Cho\\
Department of Mathematics\\ 
Chonnam National University\\ 
Gwangju 61186, Korea\\
E-mail:  jtcho@chonnam.ac.kr \vskip.2cm

\item[] Makoto Kimura\\ 
Department of Mathematics, \\
Faculty of Science, \\ 
Ibaraki University, \\ 
Mito, Ibaraki 310-8512, JAPAN\\ 
e-mail: makoto.kimura.geometry@vc.ibaraki.ac.jp

\item[] Miguel Ortega \\
Institute of Mathematics-UGR, IMAG \\
Departament of Geometry and Topology\\
Faculty of Sciences, University of Granada\\ 
18071 Granada (Spain)\\
e-mail: miortega@ugr.es

\end{trivlist} 
\end{document}